\documentclass[]{amsart}

\usepackage{mathrsfs}
\usepackage{amsmath}
\usepackage{amssymb}
\usepackage{amscd}
\usepackage{appendix}
\usepackage{amsthm}
\usepackage{amsfonts}
\usepackage[all]{xy}
\usepackage{array}
\usepackage{graphicx}
\usepackage{hyperref}
\usepackage{longtable}
\usepackage[OT2,T1]{fontenc}
\DeclareSymbolFont{cyrletters}{OT2}{wncyr}{m}{n}
\DeclareMathSymbol{\Sha}{\mathalpha}{cyrletters}{"58}

\theoremstyle{plain}
\newtheorem{thm}{Theorem}[section]
\newtheorem{lem}[thm]{Lemma}
\newtheorem{prop}[thm]{Proposition}
\newtheorem{cor}[thm]{Corollary}

\theoremstyle{definition}

\theoremstyle{remark}
\newtheorem{rem}[thm]{Remark}

\renewcommand{\AA}{\mathbb{A}}
\newcommand{\FF}{{\mathbb F}}

\newcommand{\CC}{{\mathbb C}}

\newcommand{\dfk}{\mathfrak{d}}

\newcommand{\Pcal}{\mathcal{P}}

\newcommand{\Tcal}{\mathcal{T}}

\newcommand{\Dcal}{\mathcal{D}}

\newcommand{\Kcal}{\mathcal{K}}

\newcommand{\Wcal}{\mathcal{W}}
\newcommand{\Zcal}{\mathcal{Z}}

\newcommand{\Pscr}{\mathscr{P}}

\newcommand{\tr}{\operatorname{Tr}}

\newcommand{\Nr}{\operatorname{Nr}}
\newcommand{\N}{\operatorname{N}}
\newcommand{\Mat}{\operatorname{Mat}}

\newcommand{\ord}{\operatorname{ord}}

\newcommand{\GL}{\operatorname{GL}}
\newcommand{\B}{\operatorname{B}}
\newcommand{\Z}{\operatorname{Z}}
\newcommand{\T}{\operatorname{T}}
\renewcommand{\U}{\operatorname{U}}
\newcommand{\SL}{\operatorname{SL}}

\newcommand{\Oo}{\operatorname{O}}

\newcommand{\GO}{\operatorname{GO}}
\newcommand{\SO}{\operatorname{SO}}
\newcommand{\PGL}{\operatorname{PGL}}

\newcommand{\Ad}{\operatorname{Ad}}

\newcommand{\Hom}{\operatorname{Hom}}

\newcommand{\Div}{\operatorname{Div}}

\newcommand{\re}{\operatorname{Re}}

\numberwithin{equation}{section}

\begin{document}

\title{Waldspurger formula over function fields}
\author[Chih-Yun Chuang]{Chih-Yun Chuang}
\address{Department of Mathematics, National Taiwan University, Taiwan}
\email{cychuang@ntu.edu.tw}

\author[Fu-Tsun Wei]{Fu-Tsun Wei}
\address{Department of Mathematics, National Central University, Taiwan}
\email{ftwei@math.ncu.edu.tw}

\thanks{This work was supported by grants from Ministry of Science and Technology, Taiwan.}

\maketitle


\begin{abstract}
In this paper, we derive a function field version of the Waldspurger formula for the central critical values of the Rankin-Selberg $L$-functions. This formula states that the central critical $L$-values in question can be expressed as the \lq\lq ratio\rq\rq\ of the global toric period integral to the product of the local toric period integrals. Consequently, this result provides a necessary and sufficient criterion for the non-vanishing of these central critical $L$-values, and supports the Gross-Prasad conjecture for $\SO(3)$ over function fields.
\end{abstract}


\section*{Introduction}

In 1985, Waldspurger \cite{Wa3} established a fundamental formula for the central critical value of the Rankin-Selberg $L$-function associated to an automorphic cuspidal representation of $\GL_2$ over a given number field $F$ convolved with a Hecke character on the idele class group of a quadratic field extension over $F$. 
This formula asserts that \lq\lq global toric period integrals\rq\rq\ can be written as the central critical $L$-value in question multiplying the product of \lq\lq local toric period integrals.\rq\rq\ 
From this result, 
these critical $L$-values now have been studied extensively over number fields and lead to plenty of arithmetic consequences (cf.\ \cite{CH1}, \cite{CH2}, and \cite{YZZ}). 
The main purpose of this paper is to derive a function field analogue of Walspurger's formula. \\

Let $k$ be a global function field with \textbf{odd} characteristic, and denote the adele ring of $k$ by $k_\AA$. Let $\Dcal$ be a quaternion algebra over $k$, and $K$ be a separable quadratic algebra over $k$ with an embedding $\iota: K \hookrightarrow \Dcal$. We put $\Dcal_\AA$ and $K_\AA$ to be the adelization of $\Dcal$ and $K$, respectively. Let  $\Pi^\Dcal$ be an infinite dimensional automorphic representation of $\Dcal_\AA^\times$ (cuspidal if $\Dcal$ is the matrix algebra) with a unitary central character $\eta$. Given a unitary Hecke character $\chi : K^\times \backslash K_\AA^\times \rightarrow \CC^\times$, suppose $\eta \cdot \chi\big|_{k_\AA^\times} =1$. Let $P_\chi^\Dcal \in \Hom_{K_\AA^\times}(\Pi^\Dcal, \chi^{-1})$ be the global toric period integral:
$$P_\chi^\Dcal(f) := \int_{K^{\!^\times} k_\AA^\times \backslash K_\AA^\times} f\big(\iota(a)\big)\chi(a) d^\times a, \quad \forall f \in \Pi^\Dcal.$$
The measure $d^\times a$ chosen here is the Tamagawa measure 
(cf.\ Section~\ref{sec 1.2}).
This then gives us a linear functional $\Pcal_\chi^\Dcal : \Pi^\Dcal \otimes \widetilde{\Pi}^\Dcal \rightarrow \CC$ 
(where $\widetilde{\Pi}^\Dcal$ is the contragredient representation of $\Pi^\Dcal$) defined by:
$$\Pcal_\chi^\Dcal(f\otimes \tilde{f}) := P_\chi^\Dcal(f)\cdot P_{\chi^{-1}}^\Dcal(\tilde{f}), \quad \forall f \otimes \tilde{f} \in \Pi^\Dcal \otimes \widetilde{\Pi}^\Dcal.$$

On the other hand, write $\Pi^\Dcal = \otimes_v \Pi^\Dcal_v$, and $\widetilde{\Pi}^\Dcal = \otimes_v \widetilde{\Pi}^\Dcal_v$.
We may assume that the identification between $\Pi^\Dcal$ (resp.\ $\widetilde{\Pi}^\Dcal$) and $\otimes_v \Pi^\Dcal_v$ (resp.\ $\otimes \widetilde{\Pi}^\Dcal_v$) satisfies the following equality:
$$\langle \cdot, \cdot \rangle_{\text{Pet}}^\Dcal = \frac{2 L(1,\Pi,\text{Ad})}{\zeta_k(2)} \cdot \prod_v \langle \cdot, \cdot \rangle_v^\Dcal: \Pi^\Dcal \times \widetilde{\Pi}^\Dcal \rightarrow \CC,$$
where:
\begin{itemize}
\item the pairing $\langle \cdot,\cdot\rangle_{\text{Pet}}^\Dcal$ is induced from the Petersson inner product (with respect to the Tamagawa measure, i.e.\ the total volume of $\Dcal^\times k_\AA^\times \backslash \Dcal_\AA^\times$ is $2$, cf.\ Section~\ref{sec 1.2}).
\item for each place $v$ of $k$, $\langle \cdot, \cdot \rangle_v^\Dcal$ is the natural duality pairing between $\Pi_v^\Dcal$ and $\widetilde{\Pi}_v^\Dcal$.
\item $\Pi$ is the automorphic cuspidal representation of $\GL_2(k_\AA)$ correspoding to $\Pi^\Dcal$ via the Jacquet-Langlands correspondence.
\item $L(s,\Pi,\text{Ad})$ is the adjoint $L$-function of $\Pi$.
\item $\zeta_k(s)$ is the Dedekind-Weil zeta function of $k$.
\end{itemize} 
Write $\chi = \otimes_v \chi_v$.
Then for each $v$,
the local toric period integral $\Pscr^\Dcal_{\chi,v} : \Pi^\Dcal_v \otimes \widetilde{\Pi}^\Dcal_v \rightarrow \CC$ is given by:
$$\Pscr^\Dcal_{\chi,v}(f_v \otimes \tilde{f}_v) := 
\ast \cdot \int_{K_v^\times/k_v^\times} \langle \Pi_v^\Dcal\big(\iota(a_v)\big)f_v, \tilde{f}_v\rangle_v \chi_v(a_v) d^\times a_v.$$
Here $d^\times a_v$ is the Tamagawa measure on $K_v^\times / k_v^\times$ 
(chosen in Section~\ref{sec 1.2}), and $\ast$ is a product of \lq\lq local $L$-factors\rq\rq\ so that $\Pscr^\Dcal_{\chi,v}(f_v \otimes \tilde{f}_v) = 1$ when $v$ is \lq\lq good\rq\rq\ (cf.\ Lemma~\ref{lem 5.1}). 
These local toric period integrals induce another linear functional $\Pscr^\Dcal_\chi := \otimes \Pscr^\Dcal_{\chi,v} : \Pi^\Dcal \otimes \widetilde{\Pi}^\Dcal \rightarrow \CC$.
We now state the main theorem of this paper as follows (cf.\ Theorem~\ref{thm 5.1}):

\begin{thm}\label{thm 0.1}
Under the above assumptions, we have
$$\Pcal^\Dcal_\chi = L(\frac{1}{2},\Pi \times \chi) \cdot \Pscr^\Dcal_\chi,$$
where $L(s,\Pi\times \chi)$ is the Rankin-Selberg $L$-function associated to $\Pi$ and $\chi$.
\end{thm}

We remark that $L(s,\Pi\times \chi)$ can be identified with $L(s,\Pi_K \otimes \chi)$, the $L$-function of $\Pi_K$ twisted by $\chi$, where $\Pi_K$ is Jacquet's lifting of $\Pi$ to $\GL_2(K_\AA)$ (cf.\ \cite[Theorem 20.6]{Jac}). \\

Let $\varsigma_K$ be the quadratic Hecke character of $K/k$ and put $\varsigma_{K,v}:= \varsigma_K \big|_{k_v^\times}$. From the work of Tunnell \cite{Tun} and Waldspurger \cite[Lemme 10]{Wa3}, the local toric period integral $\Pscr^\Dcal_{\chi,v}$ is not trivial if and only if
$$\quad \quad \quad \quad \quad \quad \epsilon_v(\Pi \times \chi) = \eta_v(-1) \varsigma_{K,v}(-1) \epsilon_v(\Dcal). \quad \quad \quad \quad \quad \quad (\star) $$
Here $\epsilon_v(\Pi\times \chi)$ is the local root number of $L(s,\Pi\times \chi)$ at $v$ and $\epsilon_v(\Dcal)$ is the Hasse invariant of $\Dcal$ at $v$.
This leads us to the following consequence.

\begin{cor}\label{cor 0.2}
Suppose $\prod_v \epsilon_v(\Pi\times \chi) = 1$. 
Let $\Dcal$ be the unique (up to isomorphism) quaternion algebra over $k$ so that the equality~$(\star)$ holds for every place $v$ of $k$. 
Then the non-vanishing of $L(1/2,\Pi\times \chi)$ is equivalent to the existance of an automorphic form $f \in \Pi^\Dcal$ so that
$$P_\chi^\Dcal(f) = \int_{K^{\!^\times} k_\AA^\times \backslash K_\AA^\times} f\big(\iota(a)\big) \chi(a) d^\times a \neq 0.$$
\end{cor}

In particular, via the isomorphism $\PGL_2 \cong \SO(3)$, Corollary~\ref{cor 0.2} supports the Gross-Prasad conjecture for the $\SO(3)$ case over function fields (cf.\ \cite{I-I}).\\

The proof of Theorem~\ref{thm 0.1} basically follows Waldspurger's approach in \cite{Wa3} for the number field case.
Suppose first that $K$ is a quadratic field over $k$. Let $(V_\Dcal, Q_{V_\Dcal})$ be the quadratic space $(\Dcal, \Nr_{\Dcal/k})$, where $\Nr_{\Dcal/k}$ is the reduced norm from $\Dcal$ to $k$. Given $\phi \in \Pi$ and a Schwartz function $\varphi \in S(V_\Dcal(k_\AA))$, suppose $\phi$ and $\varphi$ are both pure tensors. From the Rankin-Selberg method, we have (cf.\ Corollary~\ref{cor 3.3} (2))
\begin{eqnarray}\label{eqn 0.1}
L(2s,\varsigma_K) \cdot \Zcal(s;\phi,\varphi) &=& L(s,\Pi\times \chi) \cdot \prod_v \Zcal_v^o(s;\phi_v,\varphi_v),
\end{eqnarray}
where the zeta integral $\Zcal(s;\phi,\varphi)$ (resp.\ $\Zcal_v^o(s;\phi_v,\varphi_v)$) is defined in the beginning of Section~\ref{sec 3.2} (resp.\ Corollary~\ref{cor 3.3} (2)).
Applying the Siegel-Weil formula in Theorem~\ref{thm 3.1} and the seesaw identity (cf.\ the diagram~(\ref{seesaw})), we may connect $L(1,\varsigma_K) \cdot \Zcal(1/2;\phi,\varphi)$ with a global toric period integral $\Tcal(\phi,\varphi)$ (cf.\ the equation~(\ref{eqn 4.2}) and Proposition~\ref{prop 4.3}). On the other hand, the local zeta integral $\Zcal_v^o(1/2;\phi_v,\varphi_v)$ can be rewritten as a local toric period integral $\Tcal_v(\phi_v,\varphi_v)$ (cf.\ Proposition~\ref{prop 4.1}). 
The global (resp.\ local) Shimizu correspondence in Theorem~\ref{thm 2.2} (resp.\ Section~\ref{sec 2.3.1}) then enables us to connect $\Tcal$ (resp.\ $\Tcal_v$) with $\Pcal_\chi^\Dcal$ (resp.\ $\Pscr_{\chi,v}^\Dcal$), which completes the proof. 
Note that in our approach, we always take the original Schwartz functions (i.e.\ functions in $S(V(k_\AA))$, cf.\ Section~\ref{sec 2}), instead of using the \lq\lq extended ones\rq\rq\ (i.e.\ functions in $S(V(k_\AA) \times k_\AA^\times)$ as in \cite[Section 3]{Wa3}. This simplifies the arguments.

One ingredient of the above proof is to decompose the global Shimizu correspondence as the tensor product of local ones (cf.\ Section~\ref{sec 2.3.1} and Appendix~\ref{sec A}). To achieve this, we need to verify the Siegel-Weil formula for the dual pair $(\widetilde{\SL}_2, O(\Dcal^o))$, where $\widetilde{\SL}_2$ is the metaplectic cover of $\SL_2$, and $\Dcal^o$ consists of all the pure quaternions in $\Dcal$ (cf.\ Appendix~\ref{sec B}).

When $K = k\times k$, the existance of the embedding $\iota: K \hookrightarrow \Dcal$ forces that $\Dcal = \Mat_2$. We may write $\chi = \chi_1 \times \chi_2$ where $\chi_i$ are unitary Hecke characters on $k^\times \backslash k_\AA^\times$. In this case we have $$L(s,\Pi \times \chi) = L(s,\Pi \otimes \chi_1) \cdot L(s,\Pi \otimes \chi_2).$$ Note that the assumption $\eta \cdot \chi\big|_{k_\AA^\times} = 1$ says that $\Pi \otimes \chi_2 = \widetilde{\Pi} \otimes \chi_1^{-1}$. The global (resp.\ local) toric period integrals can then be easily identified with the product of the special values of the global (resp.\ local) zeta integrals of forms in $\Pi \otimes \chi_1$ and $\widetilde{\Pi}\otimes \chi_1^{-1}$ at $s=1/2$. Therefore Theorem~\ref{thm 0.1} follows immediately (cf.\ Appendix~\ref{sec C}).\\

Identifying $\widetilde{\Pi}^\Dcal$ with the space $\{\bar{f}: f\in \Pi^\Dcal\}$ via the Petersson inner product on $\Pi^\Dcal$,
we put $\Vert f \Vert_{\text{Pet}}^\Dcal:= \langle f, \bar{f}\rangle_{\text{Pet}}^\Dcal$ (resp.\ $\Vert f_v \Vert_{v}^\Dcal:=  \langle f_v, \bar{f}_v\rangle_v^\Dcal$).  
For non-zero pure tensors $\phi = \otimes_v \phi_v \in \Pi$ and $f = \otimes_v f_v \in \Pi^\Dcal$, 
from Theorem~\ref{thm 0.1} we obtain that
\begin{eqnarray}\label{eqn 0.2}
\frac{|P_\chi^\Dcal(f)|^2}{\Vert f\Vert_{\text{Pet}}^\Dcal}
&=& \frac{L(1/2,\Pi \times \chi)}{\Vert \phi \Vert_{\text{Pet}}^{\Mat_2}} \cdot \prod_v \alpha_v(\phi_v,f_v),
\end{eqnarray}
where
\begin{eqnarray}
&& \alpha_v(\phi_v,f_v) \nonumber  \\
&:=& \left(\frac{L_v(1,\Pi,\text{Ad})}{\zeta_v(2)} \Vert \phi_v \Vert_v^{\Mat_2}\right) \cdot
\left(\frac{L_v(1,\varsigma_K)}{L_v(1/2,\Pi\times \chi)} \int_{K_v/k_v^\times} \frac{\langle \Pi_v^\Dcal\big(\iota(a_v)\big) f_v, \bar{f}_v\rangle_v}{\Vert f_v \Vert_v^\Dcal} \chi_v(a_v) d^\times a_v\right). \nonumber
\end{eqnarray}
Taking suitable $\phi$ and $f$, it is possible to calculate the local quantities $\alpha_v(\phi_v,f_v)$ in concrete terms. Therefore the equality~(\ref{eqn 0.2}) leads us to an explicit formula of $L(1/2,\Pi\times \chi)$. This will be studied in a subsequent paper.\\

The content of this paper is given as follows. In Section~\ref{sec 1}, we first set up basic notations used throughout this paper, and fix all the Haar measures in the paper to be the Tamagawa measures. In Section~\ref{sec 2}, we recall needed properties of theta series associated to quadratic fields and quaternion algebras, and state the Shimizu correspondence in the version used here. In Section~\ref{sec 3}, we apply the Rankin-Selberg method to show the equation~(\ref{eqn 0.1}). In Section~\ref{sec 4}, we first rewrite $\Zcal_v^o(1/2;\phi_v,\varphi_v)$ in terms of the local toric period integral $\Tcal_v(\phi_v,\varphi_v)$ associated to $\phi_v$ and $\varphi_v$ in Section~\ref{sec 4.1}. Applying the seesaw identity, the special value $L(1,\varsigma_K) \cdot \Zcal(1/2;\phi,\varphi)$ equals to the global toric period integral $\Tcal(\phi,\varphi)$ associated to $\phi$ and $\varphi$ in Section~\ref{sec 4.2}. We thereby arrive at the main theorem in Section~\ref{sec 5} by applying the global and local Shimizu correspondence. In Appendix~\ref{sec A}, we recall the decomposition of the global Shimizu correspondence into the tensor product of local ones. In Appendix~\ref{sec B}, we verify the Siegel-Weil formula for the dual pair $(\widetilde{\SL}_2, O(\Dcal^o))$, where $\widetilde{\SL}_2$ is the metaplectic cover of $\SL_2$, and $\Dcal^o$ consists of all the pure quaternions in a division quaternion algebra $\Dcal$. The case when $K = k\times k$ for Theorem~\ref{thm 0.1} is proven in Appendix~\ref{sec C}.


\section{Prelimilaries}\label{sec 1}

\subsection{Basic settings}\label{sec 1.1}

Give a ring $R$, the multiplicative group of $R$ is denoted by $R^\times$. 
By $\#(S)$ for each set $S$, 
we mean the cardinality of $S$. 
\\

Let $k$ be a global function field with finite constant field $\FF_q$. 
Throughout this paper, we always assume $q$ to be {\bf odd}. 
For each place $v$ of $k$, let $k_v$ be the completion of $k$ at $v$, 
and $O_v$ be the valuation ring in $k_v$. 
Choose a uniformizer $\varpi_v$ once and for all.
Set $\FF_v:= O_v/\varpi_v O_v$, the residue field at $v$, 
and put $q_v:= \#(\FF_v)$. 
The valuation on $k_v$ is denoted by $\ord_v$, 
and we normalize the absolute value $|\cdot|_v$ on $k_v$ by $|a_v|_v := q_v^{-\ord_v(a_v)}$ for every $a_v \in k_v$.\\

Let $k_\AA$ be the ring of adeles of $k$, 
i.e.\ $k_\AA = \prod_v' k_v$, the restricted direct product of $k_v$ with respect to $O_v$.
The maximal compact subring of $k_\AA$ is denoted by $O_{\AA}$.
The group of ideles of $k$ is $k_\AA^\times$, with the maximal compact subgroup $O_{\AA}^\times$.
For $a = (a_v)_v \in k_\AA^\times$, we put $|a|_\AA := \prod_v |a_v|_v$. \\

Finally,
fix a non-trivial additive character $\psi: k_\AA \rightarrow \CC^\times$ which is trivial on $k$.
For each place $v$ of $k$, put $\psi_v := \psi\big|_{k_v}$.
Let $\delta_v$ be the \lq\lq conductor\rq\rq\ of $\psi_v$, 
i.e.\ $\psi_v$ is trivial on $\varpi_v^{-\delta_v}O_v$ but not trivial on $\varpi_v^{-\delta_v-1}O_v$.
Then $\sum_v \delta_v \cdot \deg v = 2 g_k -2$, 
where $g_k$ is the genus of $k$.

\subsection{Tamagawa measures}\label{sec 1.2}


For each place $v$ of $k$, 
choose the self-dual Haar measure $dx_v$ on $k_v$ with respect to the fixed additive character $\psi_v$, 
i.e.\ $\text{vol}(O_v, dx_v) = q_v^{-\delta_v/2}$. 
The Haar measure $dx = \prod_v dx_v$ on $k_\AA$ is then self-dual with respect to $\psi$, 
and $\text{vol}(k\backslash k_\AA, dx) = 1$. 
For the multiplicative group $k_v^\times$, 
we take the Haar measure
$$d^\times x_v:= \zeta_v(1) \cdot \frac{dx_v}{|x_v|_v},$$
where $\zeta_v(s) = (1-q_v^{-s})^{-1}$ is the local zeta function of $k$ at $v$.
Then $\text{vol}(O_v^\times, d^\times x_v) = q_v^{-\delta_v/2}$.
This gives us a Haar measure $d^\times x = \prod_v d^\times x_v$ on $k_\AA^\times$.\\

Given a separable quadratic algebra $K$ over $k$, let $\T_{K/k}$ and $\N_{K/k}$ be the trace and norm from $K$ to $k$, respectively.
Put $K_v := K\otimes_k k_v$. The Haar measures on $K_v$ and $K_v^\times$ are chosen as above for each place $v$ of $k$ (with respect to the character $\psi_v \circ \T_{K/k}$).
This induces a Haar measure $d^\times h_v$ on $K_v^\times/ k_v^\times$, and one has $\text{vol}(O_{K_v}^\times/O_v^\times,d^\times h_v) = q_v^{-(\ord_v(\dfk_K)+\delta_v)/2}$, where $\dfk_K \in \Div(k)$ is the discriminant divisor of $K$ over $k$. 
Let $K_\AA := K \otimes_k k_\AA$. We then take the Haar measure on $K_\AA^\times/k_\AA^\times$ to be $d^\times h:= \prod_v d^\times h_v$. Let $\varsigma_K$ be the quadratic character of $K/k$, i.e.\ 
$\varsigma_K: k^\times \backslash k_\AA^\times \rightarrow \CC^\times$ is the character with the kernel precisely equal to $k^\times \cdot \N_{K/k}(K_\AA^\times)$.
When $K$ is a field, one has $$\text{vol}(K^\times \backslash K_\AA^\times /k_\AA^\times, d^\times h) = 2 \cdot L(1, \varsigma_K).$$

By Hilbert's theorem 90, we may identify $K^\times/ k^\times$ with $K^1:= \{ a \in K^\times \mid \N_{K/k}(a) = 1\}$. 
Thus the chosen Haar measure $d^\times h$ on $K_\AA^\times/k_\AA^\times$ can be identified with a Haar measure $d^\times h^1$ on $K_\AA^1$.
In particular, for each place $v$ of $k$, we have
$$\text{vol}(O_{K_v}^1, d^\times h_v^1) = (\ord_v(\dfk_K)+1) \cdot q_v^{-(\ord_v(\dfk_K)+\delta_v)/2}.$$
${}$

Given a quaternion algebra $\Dcal$ over $k$, let $\tr_{\Dcal/k}$ and $\Nr_{\Dcal/k}$ be the reduced trace and norm from $\Dcal$ to $k$, respectively. Put $\Dcal_v := \Dcal\otimes_k k_v$ for each place $v$ of $k$. The Haar measure $db_v$ on $\Dcal_v$ for each $v$ is taken to be self-dual with respect to $\psi_v \circ \tr_{\Dcal/k}$. For the multiplicative group $\Dcal_v^\times$, we choose
$$d^\times \tilde{b}_v := \zeta_v(1) \cdot \frac{d b_v}{|\Nr_{\Dcal/k}(b_v)|_v}.$$
Globally, put $\Dcal_\AA:= \Dcal \otimes_k k_\AA$. We choose the Haar measure $d^\times b$ on $\Dcal_\AA^\times$ satisfying that for each maximal compact open subgroup $\Kcal = \prod_v \Kcal_v \subset \Dcal_\AA^\times$,
one has
$$\text{vol}(\Kcal , d^\times b) := \prod_v \text{vol}(\Kcal_v , d^\times \tilde{b}_v).$$
Via the exact sequence 
$$1 \rightarrow \Dcal^1 \rightarrow \Dcal^\times \rightarrow k^\times \rightarrow 1$$
the chosen Haar measures $d^\times b$ on $\Dcal_\AA^\times$ and $d^\times x$ on $k_\AA^\times$ determine a Haar measure $d^\times b_1$ on $\Dcal_\AA^1$.
Moreover, it is known that 
(cf.\ \cite[Theorem 3.3.1]{We2})
$$\text{vol}(\Dcal^\times k_\AA^\times \backslash \Dcal_\AA^\times, d^\times b) = 2 \quad \text{ and } \quad
\text{vol}(\Dcal^1 \backslash \Dcal_\AA^1, d^\times b_1) = 1.$$

\section{Theta series}\label{sec 2}

\subsection{Weil representation}\label{sec 2.1}

Let $(V, Q_V)$ be a non-degenerate quadratic space over $k$ with even dimension (then $\dim_k V \leq 4$).
Set $$\langle x,y\rangle_V:= Q_V(x+y)-Q_V(x)-Q_V(y), \quad \forall x,y \in V,$$ 
the bilinear form associated to $Q_V$.
Given an arbitrary $k$-algebra $R$, 
set $V(R):= V \otimes_k R$. 
For our purpose, the (local) Weil representation $\omega^V_v$ of $\big(\!\SL_2 \times \Oo(V)\big)(k_v)$ on the Schwartz space $S(V(k_v))$ is chosen with respect to $\overline{\psi}_v$ for every place $v$ of $k$.
We denote by $\omega^V:= \otimes_v \omega^V_v$ the (global) Weil representation of $\big(\!\SL_2 \times \Oo(V)\big)(k_\AA)$ on the Schwartz space $S(V(k_\AA))$.\\

Let $\GO(V)$ be the orthogonal similitude group of $V$ over $k$. 
Put
$$[\GL_2 \times \GO(V)]:=\{(g,h) \in \GL_2\times \GO(V)\mid \det(g) = \nu(h)\}.$$
Here $\nu(h)$ is the factor of similitude for $h \in \GO(V)$.
We extend $\omega^V$ to a representation (still denoted by $\omega^V$) of $[\GL_2\times \GO(V)](k_\AA)$ on $S(V(k_\AA))$ by the following: 
for every pair $(g,h) \in [\GL_2\times \GO(V)](k_\AA)$ and $\varphi \in S(V(k_\AA))$, 
set
$$\big(\omega^V(g,h)\varphi\big)(x):= |\det(g)|_{\AA}^{-\frac{1}{2}} \cdot  \big(\omega^V(\begin{pmatrix}1&0\\0&\det(g)^{-1}\end{pmatrix}g)\varphi\big)(h^{-1}x), \quad \forall x \in V(k_\AA).$$
Given $(g,h) \in [\GL_2\times \GO(V)](k_\AA)$ and $\varphi \in S(V(k_\AA))$, let
$$\theta^V(g,h;\varphi):= \sum_{x \in V(k)} \big(\omega^V(g,h)\varphi\big)(x).$$
For every $\varphi \in S(V(k_\AA))$, the theta series $\theta^V(\cdot,\cdot;\varphi)$ 
is invariant by $[\GL_2\times \GO(V)](k)$ via left multiplications.

\subsection{Quadratic theta series}\label{sec 2.2}

Let $K$ be a quadratic field extension of $k$. 
Given $\gamma \in k^\times$, 
let $(V_{(\gamma)}, Q_{(\gamma)}):= (K, \gamma \cdot \N_{K/k})$, 
where $\N_{K/k}$ is the norm form on $K/k$. 
Then one has $\GO(V_{(\gamma)}) \cong K^\times \rtimes \langle \tau_K\rangle$, 
where $\tau_K(x):= \bar{x}$ for every $x \in K = V_{(\gamma)}(k)$. 
We may identify 
$K^1:= \{h \in K\mid \N_{K/k}(h) = 1\}$ with the special orthogonal group $\SO(V_{(\gamma)})$. \\

Let $\GL_2^{{}^{\!\!+_{\!\tiny K}}}$ be the image of natural projection of $[\GL_2\times \GO(V_{(\gamma)})]$ into $\GL_2$.
Given a unitary Hecke character $\chi$ on $K^\times \backslash K_\AA^\times$ and $\varphi \in S(V_{(\gamma)}(k_\AA))$, 
set
$$\theta_\chi^{(\gamma)}(g; \varphi):= 
\int_{K^1 \backslash K_\AA^1} \theta^{V_{(\gamma)}}(g, r h_g; \varphi) \chi( r h_g) dr, \quad \forall g \in \GL_2^{{}^{\!\!+_{\!\tiny K}}}(k_\AA).$$
Here $h_g \in K_\AA^\times$ is chosen so that $\N_{K/k}(h_g) = \det(g)$.
Then $\theta_\chi^{(\gamma)}(\cdot ; \varphi)$ is invariant under $\GL_2^{{}^{\!\!+_{\!\tiny K}}}(k)$ by left multiplications, 
and has a central character equal to $\varsigma_K \cdot \chi\big|_{k_\AA^\times}$, 
where $\varsigma_K$ is the quadratic Hecke character of $K/k$.
When $\gamma = 1$, we will denote by $\theta^K(\cdot,\cdot;\varphi)$ and $\theta_\chi^K(\cdot;\varphi)$ the quadratic theta series
$\theta^{V_{(1)}}(\cdot,\cdot;\varphi)$ and $\theta_\chi^{(1)}(\cdot;\varphi)$, respectively.

\subsubsection{Whittaker functions}\label{sec 2.2.1}

Given $\gamma \in k^\times$,
the Whittaker function (with respect to $\overline{\psi}$) attached to $\theta_\chi^{(\gamma)}(\cdot;\varphi)$ for $\varphi \in S(V_\gamma(k_\AA))$ is:
$$W_\chi^{(\gamma)}(g;\varphi):= \int_{k\backslash k_\AA} \theta_\chi^{(\gamma)} \left( \begin{pmatrix} 1 & n \\ 0 & 1\end{pmatrix} g;\varphi\right)\psi(n) dn.$$
Then
$$W_\chi^{(\gamma)}\left(\begin{pmatrix}1&n\\0&1\end{pmatrix}g;\varphi\right) = \overline{\psi(n)}\cdot W_\chi^{(\gamma)}(g;\varphi), \quad \forall g \in \GL_2^{{}^{\!\!+_{\!\tiny K}}}(k_\AA) \text{ and } n \in k_\AA.$$
It is straightforward that:

\begin{lem}\label{lem 2.1}
Suppose $\varphi = \otimes_v \varphi_v \in S(V_\gamma(k_\AA))$ is a pure tensor. 
Then $W_\chi^{(\gamma)}(\cdot ;\varphi)$ is factorizable. More precisely, 
for $g = (g_v)_v \in \GL_2^{{}^{\!\!+_{\!\tiny K}}}(\AA_k)$, choose $h_g = (h_{g,v})_v \in K_\AA^\times$ so that $\det(g) = \N_{K/k}(h_g)$.
One has $W_\chi^{(\gamma)}(g ;\varphi) = \prod_v W_{\chi,v}^{(\gamma)}(g_v ;\varphi_v)$, where
$$
W_{\chi,v}^{(\gamma)}(g_v ;\varphi_v) := \int_{K_v^1} \big(\omega^{V_{(\gamma)}}_v(g_v, r_v h_{g,v})\varphi_v\big)(1) \cdot \chi_v(r_v h_{g,v}) dr_v.
$$
\end{lem}

\subsection{Quaternionic theta series}\label{sec 2.3}

Let $\Dcal$ be a quaternion algebra over $k$, 
and denote by $\Nr_{\Dcal/k}$ (resp.\ $\tr_{\Dcal/k}$) the reduced norm (resp.\ trace) on $\Dcal/k$. 
Let $(V_\Dcal,Q_{V_\Dcal}):= (\Dcal,\Nr_{\Dcal/k})$. 
Then we have the following exact sequence:
$$
1\longrightarrow k^\times \longrightarrow  (\Dcal^\times \times \Dcal^\times) \rtimes \langle \tau_\Dcal \rangle \longrightarrow \GO(V_\Dcal) \longrightarrow 1.
$$ 
Here:
\begin{itemize}
\item $k^\times$ embeds into $\Dcal^\times \times \Dcal^\times$ diagonally; 
\item every pair $(b_1,b_2) \in \Dcal^\times \times \Dcal^\times$ is sent to 
$$[b_1,b_2]:= (x \mapsto b_1 x b_2^{-1},\ x \in \Dcal) \in \GO(V_\Dcal);$$
\item $\tau_\Dcal(x):= \bar{x} = \tr_{\Dcal/k}(x)-x$ for every $x \in \Dcal$.
\end{itemize}

Let $\Pi^\Dcal$ be an infinite dimensional automorphic representation of $\Dcal_\AA^\times$ which is cuspidal if $\Dcal = \Mat_2$. Suppose the central character of $\Pi^\Dcal$ is unitary. Let $\Pi$ be the automorphic cuspidal representation of $\GL_2(k_\AA)$ corresponding to $\Pi^\Dcal$ via the Jacquet-Langlands correspondence. Given $\varphi \in S(V_\Dcal(k_\AA))$ and $\phi \in \Pi$, for $b_1,b_2 \in \Dcal_\AA^\times$ we set
$$\theta^\Dcal(b_1,b_2;\phi,\varphi):= \int_{\SL_2(k)\backslash \SL_2(k_\AA)} 
\phi \big(g^1 \alpha(b_1b_2^{-1})\big) \cdot \theta^{V_\Dcal}\big(g^1 \alpha(b_1b_2^{-1}), [b_1,b_2]; \varphi\big) dg^1.$$
Here $\alpha(b):= \begin{pmatrix}1&0\\0&\Nr_{\Dcal/k}(b) \end{pmatrix}$ for every $b \in \Dcal_\AA^\times$, and $dg^1$ is the Tamagawa measure on $\SL_2(k_\AA)$ (cf.\ Section~\ref{sec 1.2}).
It is clear that $\theta^\Dcal(\cdot,\cdot;\phi,\varphi)$ is invariant by $\Dcal^\times \times \Dcal^\times$ via left multiplications.
Put $$\Theta^\Dcal(\Pi):= \left\{\theta^\Dcal(\cdot,\cdot;\phi , \varphi)\mid \phi \in \Pi, \varphi \in S(V_\Dcal(k_\AA))\right\}.$$
The Shimizu correspondence says (cf.\ \cite[Theorem 1]{Shi}):

\begin{thm}\label{thm 2.2}
Given an infinite dimensional automorphic representation $\Pi^\Dcal$ of $\Dcal_\AA^\times$ (cuspidal if $\Dcal = \Mat_2$), suppose the central character of $\Pi^\Dcal$ is unitary. Then 
\begin{eqnarray}\label{eqn Sh}
\Theta^\Dcal(\Pi) &=& \{ f_1 \otimes \bar{f}_2 : \Dcal_\AA^\times \times \Dcal_\AA^\times \rightarrow \CC \mid f_1, f_2 \in \Pi^\Dcal \}_{\CC-\text{\rm span}}.
\end{eqnarray}
Here $f_1\otimes \bar{f}_2\ (b,b') := f_1(b) \cdot \bar{f}_2(b')$ for every $b,b' \in \Dcal_\AA^\times$.
Consequently, let $\widetilde{\Pi}^\Dcal$ be the contragredient representation of $\Pi^\Dcal$. Identifying $\widetilde{\Pi}^\Dcal$ with the space $\{ \bar{f} \mid f \in \Pi^\Dcal\}$ via the Petersson inner product, the equality~\text{\rm (\ref{eqn Sh})} induces an isomorphism
$$
\text{\bf Sh}: \Theta^\Dcal(\Pi) \cong \Pi^\Dcal \otimes \widetilde{\Pi}^\Dcal.
$$
\end{thm}

\subsubsection{Local Shimizu correspondence}\label{sec 2.3.1}

We may identify $\Pi$ with $\otimes_v \Pi_v$ naturally via the Whittaker model of $\Pi$ (with respect to $\psi$). 
Let $v$ be a place of $k$. For $\phi_v \in \Pi_v$ and $\varphi_v \in S(V_\Dcal(k_v))$, put
\begin{eqnarray}
&&\theta_v^{\Dcal,o}(b_v,b_v';\phi_v,\varphi_v) \nonumber \\
&:=&
\frac{\zeta_v(2)}{L_v(1,\Pi,\text{Ad})} \cdot
\int_{\U(k_v) \backslash \SL_2(k_v)} W_{\phi_v}\big(g_v^1 \alpha(b_v b_v'^{-1})\big) \cdot 
\big(\omega^\Dcal_v(g_v^1 \alpha(b_v b_v'^{-1}), [b_v,b_v'])\varphi_v\big)(1) d g_v^1. \nonumber
\end{eqnarray}
Here $W_{\phi_v}$ is the Whittaker function of $\phi_v$ (with respect to $\psi_v$), the map $\alpha$ is defined in the above of Theorem~\ref{thm 2.2}, and $U \subset \SL_2$ is the standard unipotent subgroup. Observe that when $v$ is \lq\lq good\rq\rq\ we have $\theta_v^{\Dcal,o}(b_1,b_v';\phi_v,\varphi_v) =1$ (cf.\ Theorem~\ref{thm A.2} (1)). 
Moreover, for pure tensors $\phi = \otimes_v \phi_v \in \Pi$ and $\varphi = \otimes_v \phi_v \in S(V_\Dcal(k_\AA))$ we have
(cf.\ Theorem~\ref{thm A.2})
\begin{eqnarray}\label{eqn 2.2}
&& \int_{\Dcal^\times k_\AA^\times \backslash \Dcal_\AA^\times} \theta^\Dcal(b b_1,b b_2;\phi,\varphi) d^\times b \\
&=& \frac{2 L(1,\Pi,\text{Ad})}{\zeta_k(2)} \cdot \prod_v \theta_v^{\Dcal,o}(b_{1,v},b_{2,v};\phi_v,\varphi_v), \quad \forall b_1,b_2 \in \Dcal_\AA^\times. \nonumber
\end{eqnarray}
Put $$\Theta_v^\Dcal(\Pi_v):= \{ \theta_v^{\Dcal,o}(\cdot,\cdot;\phi_v,\varphi_v)\mid \phi_v \in \Pi_v,\ \varphi_v \in S(V_\Dcal(k_v))\}.$$
Then the above equality implies that (cf.\ Proposition~\ref{prop A.L})
$$\Theta_v^\Dcal(\Pi_v) = \{ f_v \otimes \tilde{f}_v : \Dcal_v^\times \times \Dcal_v^\times \rightarrow \CC \mid f_v \in \Pi_v^\Dcal,\ \tilde{f}_v \in \widetilde{\Pi}_v^\Dcal\}_{\CC-\text{\rm span}}.$$
Here $f_v \otimes \tilde{f}_v$ is viewed as a matrix coefficient:
$$f_v \otimes \tilde{f}_v (b_v,b_v'):= \langle \Pi_v^\Dcal(b_v) f_v, \widetilde{\Pi}_v^\Dcal (b_v') \tilde{f}_v\rangle_v^\Dcal, \quad \forall b_v,b_v' \in \Dcal_v^\times,$$
where $\langle\cdot,\cdot \rangle_v^\Dcal: \Pi_v^\Dcal \times \widetilde{\Pi}_v^\Dcal \rightarrow \CC$ is the natural duality pairing.
Consequently, we have an isomorphism $\textbf{Sh}_v : \Theta_v^\Dcal(\Pi_v) \cong \Pi_v^\Dcal \otimes \widetilde{\Pi}_v^\Dcal$.

\begin{rem}\label{rem 2.3}
Let $\langle \cdot,\cdot \rangle_{\text{Pet}}^\Dcal: \Pi^\Dcal \times \widetilde{\Pi}^\Dcal \rightarrow \CC$ be the Petersson pairing.
The equality~(\ref{eqn 2.2}), together with $\textbf{Sh}$ and $\textbf{Sh}_v$, provide us a way to indentify $\Pi^\Dcal$ (resp.\ $\widetilde{\Pi}^\Dcal$) with $\otimes_v \Pi^\Dcal_v$ (resp.\ $\otimes_v \widetilde{\Pi}^\Dcal_v$) so that for pure tensors $f = \otimes_v f_v \in \Pi^\Dcal$ and $\tilde{f} = \otimes_v \tilde{f}_v \in \widetilde{\Pi}^\Dcal$, we have
$$\langle f, \tilde{f}\rangle_{\text{Pet}}^\Dcal = \frac{2 L(1,\Pi,\text{Ad})}{\zeta_k(2)} \cdot \prod_v \langle f_v, \tilde{f}_v\rangle_{v}^\Dcal.$$
\end{rem}

\section{Zeta integrals and Rankin-Selberg method}\label{sec 3}

\subsection{Siegel Eisenstein series}\label{sec 3.1}

Let $K$ be a quadratic field over $k$.
Fix $\gamma \in k^\times$. Recall that we put $(V_{(\gamma)},Q_{(\gamma)}) = (K, \gamma \cdot \N_{K/k})$.
 Given $\varphi \in S(V_{(\gamma)}(k_\AA))$, the Siegel section associated to $\varphi$ is defined by
$$\Phi_\varphi(g,s):= \frac{|a|_{\AA}^s}{|b|_{\AA}^s} \cdot \varsigma_K(b) \cdot \big(\omega^{V_{(\gamma)}}(\kappa) \varphi\big)(0)$$
for every $g = \begin{pmatrix} a& n \\ 0 & b \end{pmatrix} \kappa \in \GL_2(k_\AA)$ with $a,b \in k_\AA^\times$, $n \in k_\AA$, $\kappa \in \SL_2(O_{\AA})$, and $s \in \CC$.
Here $\varsigma_K$ is the quadratic character of $K/k$.
The Siegel Eisenstein series associated to $\varphi$ is
$$E(g,s,\varphi):= \sum_{\gamma \in \B(k)\backslash \GL_2(k)}\Phi_\varphi(\gamma g,s),\quad \forall g \in \GL_2(k_\AA),$$
which converges absolutely for $\re(s)>1$.
It is known that $E(g,s,\varphi)$ has meromorphic continuation to the whole complex $s$-plane and satisfies a functional equation with the symmetry between $s$ and $1-s$. 
Note that $E(g,s,\varphi)$ is always holomorphic at the central critical point $s=1/2$, and the following formula holds (cf.\ \cite[Theorem 0.1]{Wei1}):

\begin{thm}\label{thm 3.1}
{\rm (The Siegel-Weil formula)} Fix $\gamma \in k^\times$. Given $\varphi \in S(V_{(\gamma)}(k_\AA))$, one has
$$E(g,\frac{1}{2},\varphi) = \frac{1}{L(1,\varsigma_K)}\cdot \theta^{(\gamma)}_{\mathbf{1}_K}(g,\varphi), \quad \forall g \in \GL_2^{{}^{\!\!+_{\!\tiny K}}}(k_\AA),$$
where
$\mathbf{1}_K$ is the principal character on $K_\AA^\times$.
\end{thm}


\subsection{Zeta integrals}\label{sec 3.2}

Let $\Dcal$ be a quaternion algebra over $k$. Given a quadratic field extension $K$ over $k$ with an embedding $K\hookrightarrow \Dcal$, we write $\Dcal = K + Kj$ where $j^2 = \gamma \in k^\times$ and $j b = \bar{b} j $ for every $b \in K$.
Set $(V_\Dcal, Q_{V_\Dcal}):= (\Dcal,\Nr_{\Dcal/k})$. Then 
$$(V_{\Dcal},Q_{V_\Dcal}) = (V_{(1)}, Q_{(1)}) \oplus (V_{(-\gamma)}, Q_{(-\gamma)}).$$

Let $\Pi$ be an automorphic cuspidal representation of $\GL_2(k_\AA)$ with a unitary central character denoted by $\eta$. Given a Hecke character $\chi : K^\times \backslash K_\AA^\times \rightarrow \CC^\times$, suppose that $\chi$ is unitary and $\eta \cdot \chi\big|_{k_\AA^\times} = 1$. For $\phi \in \Pi$ and $\varphi \in S(V_\Dcal(k_\AA))$, we are interested in the following (global) zeta integral:
writing $\varphi = \sum_i \varphi_{1,i} \oplus \varphi_{2,i}$ with $\varphi_{1,i} \in S(V_{(1)}(k_\AA))$ and $\varphi_{2,i} \in S(V_{(-\gamma)}(k_\AA))$, we set
$$\Zcal(s; \phi, \varphi):= \sum_i \int_{\Z(k_\AA) \GL_2^{{}^{\!\!+_{\!\tiny K}}}(k)\backslash \GL_2^{{}^{\!\!+_{\!\tiny K}}}(k_\AA)} \phi(g) \theta_{\chi}^K(g,\varphi_{1,i}) E(g,s;\varphi_{2,i}) dg.$$
Here $\Z$ is the center of $\GL_2$, and $dg$ is the Tamagawa measure on $\GL_2(k_\AA)$ restricting to $\GL_2^{{}^{\!\!+_{\!\tiny K}}}(k_\AA)$ (cf.\ Section~\ref{sec 1.2}).
This integral is a meromorphic function on the complex $s$-plane.
Moreover, one asserts:

\begin{prop}\label{prop 3.2}
Given pure tensors $\phi = \otimes_v \phi_v \in \Pi$ and $\varphi = \otimes_v \varphi_v \in S(V_\Dcal(k_\AA))$,
one has
$$\Zcal(s; \phi, \varphi) = 
\prod_v \Zcal_v(s; \phi_v, \varphi_v),$$
where $\Zcal_v(s; \phi_v, \varphi_v)$ is equal to
$$\int_{K_v^\times} \left(\int_{\SL_2(O_v)} W_{\phi_v}\left(\begin{pmatrix}\N_{K/k}(h) & 0 \\0&1\end{pmatrix} \kappa_v^1\right)  
 \cdot \Big(\omega^\Dcal_v(\kappa_v^1)\varphi_v \Big)(\bar{h}) \, d \kappa_v^1\right) \chi_v(h) |\N_{K/k}(h)|_v^{s-\frac{1}{2}} d^\times h; $$
and $W_{\phi_v}$ is the local Whittaker function associated to $f_v$ (with respect to $\psi_v$).
\end{prop}

\begin{proof}
Without loss of generality, assume $\varphi = \varphi_1 \oplus \varphi_2$.
Let $\B^{{}^{\!\!+_{\!\tiny K}}}:= B \cap \GL_2^{{}^{\!\!+_{\!\tiny K}}} = \Z \cdot \T_1^{{}^{+_{\!\tiny K}}} \cdot \U$,
where 
$$\T_1:= \begin{pmatrix} * & 0 \\ 0&1\end{pmatrix},\ \T_1^{{}^{+_{\!\tiny K}}}:= \T_1 \cap \GL_2^{{}^{\!\!+_{\!\tiny K}}},\text{ and } \U := \begin{pmatrix} 1 & * \\ 0&1\end{pmatrix}.$$
Put $\GL_2^{{}^{\!\!+_{\!\tiny K}}}(O_{\AA}):= \GL_2(O_{\AA}) \cap \GL_2^{{}^{\!\!+_{\!\tiny K}}}(k_\AA)$. From the Iwasawa decomposition 
$$\GL_2^{{}^{\!\!+_{\!\tiny K}}}(k_\AA) = \B^{{}^{\!\!+_{\!\tiny K}}}(k_\AA) \cdot \GL_2^{{}^{\!\!+_{\!\tiny K}}}(O_{\AA}),$$
we write the zeta integral $\Zcal(s; \phi, \varphi)$ as
\begin{eqnarray}
&& \Zcal(s; \phi, \varphi) \nonumber \\
&= & \int_{\Z(O_{\AA}) \backslash \GL_2^{{}^{\!\!+_{\!\tiny K}}}(O_{\AA})}
\int_{T_1^{{}^{+_{\!\tiny K}}}(k_\AA)}
W_\phi(t \kappa^1) W_\chi^K(t \kappa^1; \varphi_1) \big(\omega^{(-\gamma)}(\kappa^1)\varphi_2\big)(0)|t|_{\AA}^{s-1} d^\times t d \kappa, \nonumber
\end{eqnarray}
where for every $\kappa \in \GL_2(O_{\AA})$, we put $\kappa^1 := \begin{pmatrix} \det(\kappa)^{-1} & 0 \\ 0 & 1\end{pmatrix} \kappa \in \SL_2(O_{\AA})$.
Note that for each place $v$ of $k$, we have the following exact sequence:
$$
\xymatrix@C=2pc@R=2pc{
1 \ar@{->}[r] & \{\pm 1\} \ar@{->}[r] & \SL_2(O_v) \ar@{->}[r]  & \Z(O_v) \backslash \GL_2(O_v) \ar@{->}[r]^{\quad \quad \det} & \displaystyle\frac{O_v^\times}{(O_v^\times)^2}  \ar@{->}[r]  & 1.
}$$
Therefore when $\varphi$ and $f$ are pure tensors, one has
$$ \Zcal(s; \phi, \varphi) =  
\prod_v \Zcal_v' (s; \phi_v,\varphi_v),$$
where
$$\Zcal_v'(s; \phi_v,\varphi_v) :=  \int_{\SL_2(O_v)}
\int_{T_1^{{}^{+_{\!\tiny K}}}(k_v)}
W_{\phi_v}(t_v \kappa^1_v) W_{\chi,v}^K(t \kappa^1_v; \varphi_{1,v}) \big(\omega_v^{(-\gamma)}(\kappa^1_v)\varphi_{2,v}\big)(0)|t|_v^{s-1} d^\times t_v d \kappa^1_v.
$$

By Lemma~\ref{lem 2.1}, the local zeta integral $\Zcal_v'(s; \phi_v,\varphi_v)$ becomes
\begin{eqnarray}
 &&\Zcal_v'(s; \phi_v,\varphi_v)  \nonumber \\
&=&  \int_{\SL_2(O_v)} 
\int_{T_1^{{}^{+_{\!\tiny K}}}(k_v)}
W_{\phi_v}(t_v \kappa^1_v) \nonumber \\
&&  \cdot \left(
\int_{K_v^1} \big(\omega_v^K(\kappa^1_v)\varphi_{1,v}\big)(\overline{r_v h_{t,v}}) \chi_v(r_v h_{t,v}) d r_v\right)
\big(\omega_v^{(-\gamma)}(\kappa^1_v)\varphi_{2,v}\big)(0)|t|_v^{s-\frac{1}{2}} d^\times t_v d \kappa^1_v
 \nonumber \\
 &=& 
 \Zcal_v(s; \phi_v,\varphi_v). \nonumber
\end{eqnarray}

\end{proof}

The following results are straightforward.

\begin{cor}\label{cor 3.3}
$(1)$ Suppose $v$ is \lq\lq good\rq\rq, i.e.\ the conductor of $\psi_v$ is trivial, $\Pi_v$ is an unramified principal series, $\phi_v \in \Pi_v$ is spherical with $W_{\phi_v}(1) = 1$, $v$ is unramified in $K$, $\chi_v$ is unramified, $\ord_v(\gamma) = 0$, and $\varphi = \varphi_1 \oplus \varphi_2$ with $\varphi_1 = \varphi_2 = \mathbf{1}_{O_{K_v}}$.
We have
$$\Zcal_v(s; \phi_v, \varphi_v) = \frac{L_v(s, \Pi \times \chi)}{L_v(2s,\varsigma_{K})}.$$
$(2)$ Given $\phi_v \in \Pi_v$ and $\varphi_v \in S(V_\Dcal(k_v))$, put $$\Zcal_v^o(s; \phi_v, \varphi_v):= \frac{L_v(2s,\varsigma_{K})}{L_v(s, \Pi\times \chi)}\cdot 
\Zcal_v(s; \phi_v, \varphi_v).$$
Then $\Zcal_v^o(s; \phi_v, \varphi_v) = 1$ for all but finitely many $v$, and
$$\Zcal(s; \phi,\varphi) =  \frac{L(s,\Pi \times \chi)}{L(2s,\varsigma_K)} \cdot  
\prod_v \Zcal_v^o(s; \phi_v, \varphi_v)$$
for every pure tensors $\phi \in \Pi$ and $\varphi \in S(V_\Dcal(k_\AA))$.\\
$(3)$ The (local) zeta integral $\Zcal_v^o(s; \phi_v, \varphi_v)$ always converges at $s=1/2$.
\end{cor}

\section{Central critical values of zeta integrals}\label{sec 4}

Let $\Dcal$, $K$, $\Pi$, $\eta$, and $\chi$ be as in the above section. For pure tensors $\phi = \otimes_v \phi_v \in \Pi$ and $\varphi = \otimes \varphi_v \in S(V_\Dcal(k_\AA))$, we shall express $\Zcal(1/2; \phi,\varphi)$ (resp.\ $\Zcal_v(1/2;\phi_v,\varphi_v)$) in terms of global (resp.\ local) \lq\lq toric period integrals\rq\rq\ of the pair $(\phi,\varphi)$ (resp.\ ($\phi_v,\varphi_v$)).

\subsection{Local case}\label{sec 4.1}

We may rewrite $\Zcal_v(1/2; \phi_v,\varphi_v)$ for $\phi_v \in \Pi_v$ and $\varphi_v \in S(V_\Dcal(k_v))$ as follows:

\begin{prop}\label{prop 4.1}
Given $\phi_v \in \Pi_v$ and $\varphi_v \in S(V_\Dcal(k_v))$, we have
\begin{eqnarray}
\Zcal_v(\frac{1}{2}; \phi_v,\varphi_v) &=& \frac{L_v(1,\Pi,\text{\rm Ad})}{\zeta_v(2)} \cdot \int_{K_v^\times/ k_v^\times} \theta_v^{\Dcal,o}( h_v,1;\phi_v,\varphi_v) \chi_v (h_v) d^\times h_v. \nonumber 
\end{eqnarray}
Here $\theta_v^{\Dcal,o}(\cdot,\cdot;\phi_v,\varphi_v)$ is defined in \text{\rm Section~\ref{sec 2.3.1}}.
\end{prop}

\begin{proof}
Given $h_v \in K_v^\times$ and $g_v^1 \in \SL_2(k_v)$, one has
$$\omega^\Dcal_v\big(g_v^1\alpha(h_v),[h_v,1]\big)\varphi_v (1)
= |\N_{K/k}(h_v)|_v^{-1} \cdot \omega_v^\Dcal\big(\alpha(h_v)^{-1} g_v^1\alpha(h_v)\big)(h_v^{-1}).$$
From the Iwasawa decomposition:
$$\SL_2(k_v) = B^1(k_v) \cdot \Big(\alpha(h_v) \SL_2(O_v) \alpha(h_v)^{-1}\Big),$$
we may write
$$dg_v^1 = |\N_{K/k}(h_v)|_v \cdot d_L b^1_v \cdot d_R \kappa_v^1.$$
Thus
\begin{eqnarray}
&& \frac{L_v(1,\Pi,\text{Ad})}{\zeta_v(2)} \cdot \theta_v^{\Dcal,o}(h_v,1;\phi_v,\varphi_v) \nonumber \\
&=& \int_{\SL_2(O_v)} \int_{k_v^\times}
W_{\phi_v}\left(\begin{pmatrix} a_v & 0 \\ 0 & a_v^{-1} \end{pmatrix} \alpha(h_v) \kappa_v^1 \right)
\left(\omega^\Dcal_v(\begin{pmatrix} a_v & 0 \\ 0 & a_v^{-1} \end{pmatrix} \kappa_v^1)\varphi_v\right)(h_v^{-1}) \frac{d^\times a_v}{|a|_v^2} d\kappa_v^1 \nonumber \\
&=&
\int_{k_v^\times} \left(\int_{\SL_2(O_v)} W_{\phi_v}\left(\begin{pmatrix} \N_{K/k}(a_v h_v)&0 \\ 0& 1\end{pmatrix} \kappa_v^1 \right)
\big(\omega^\Dcal_v (\kappa_v^1)\varphi_v\big)(a_v \overline{h_v}) d\kappa_v^1\right) \chi_v(a_v) d^\times a_v \nonumber
\end{eqnarray}
Therefore the result follows immediately.
\end{proof}

Let $\Pi^\Dcal = \otimes_v \Pi^\Dcal_v$ be, if exists, the automorphic representation of $\Dcal_\AA^\times$ corresponding to $\Pi$ via the Jacquet-Langlands correspondence.
For $\phi_v \in \Pi_v$ and $\varphi_v \in S(V_\Dcal(k_v))$, we may view $\theta_v^{\Dcal,o}(\cdot,\cdot;\phi_v,\varphi_v)$ as a matrix coefficient of $\Pi_v^\Dcal \otimes \widetilde{\Pi}_v^\Dcal$ (cf.\ Proposition~\ref{prop A.L}).
Define the local toric period integral of the pair $(\phi_v,\varphi_v)$ by
\begin{eqnarray}\label{eqn 4.1}
\,\,\,\,\,\,\,\,\,\,\,\, \Tcal_v(\phi_v,\varphi_v):=  \frac{L_v(1,\varsigma_{K}) L_v(1,\Pi, \text{Ad})}{L_v(\frac{1}{2},\Pi \times \chi) \zeta_v(2)} \cdot 
\int_{K_v^\times/ k_v^\times} \theta_v^{\Dcal,o}( h_v,1;\phi_v,\varphi_v) \chi_v (h_v) d^\times h_v.
\end{eqnarray}
Then the above proposition says 
$$\Zcal_v^o(\frac{1}{2}; \phi_v,\varphi_v) = \Tcal_v(\phi_v,\varphi_v).$$

\subsection{Global case}\label{sec 4.2}


Put
$$ [\GO(V_{(1)})\times \GO(V_{(-\gamma)})]:= \{(h_1,h_2) \in \GO(V_{(1)})\times \GO(V_{(-\gamma)}) \mid \nu(h_1) = \nu(h_2)\},$$
which is viewed as a subgroup of $\GO(V_\Dcal)$.
Note that $\GO(V_{(a)}) \cong K^\times \rtimes \langle \tau_K \rangle$ for every $a \in k^\times$, and we have the following exact sequence 
$$1\longrightarrow k^\times \longrightarrow  (\Dcal^\times \times \Dcal^\times) \rtimes \langle \tau_\Dcal \rangle \longrightarrow \GO(V_\Dcal) \longrightarrow 1.$$ 
Here $k^\times$ embeds into $\Dcal^\times \times \Dcal^\times$ diagonally, and every pair $(b_1,b_2) \in \Dcal^\times \times \Dcal^\times$ is sent to $(x \mapsto b_1 x b_2^{-1},\ x \in V = \Dcal) \in \GO(V_\Dcal)$.
Let 
\begin{eqnarray}
[K^\times\times K^\times] &:=& \{(h_1,h_2) \in K^\times \times K^\times \mid \N_{K/k}(h_1) = \N_{K/k}(h_2)\} \nonumber \\
&=& K^\times \times K^\times \cap [\GO(V_{(1)})\times \GO(V_{(-\gamma)})]. \nonumber
\end{eqnarray}
Define
$\iota : [K^\times\times K^\times] \hookrightarrow \Dcal^\times \times \Dcal^\times$ by sending $(h_1,h_2)$ to $(h_1 h', h') \in (\Dcal^\times \times \Dcal^\times)/k^\times$, where $h' \in K^\times$ such that $h'/\overline{h'} = h_2/h_1$. Then the following diagram commutes:
$$
\xymatrix@C=2pc@R=2pc{
[K^\times \times K^\times]  \ar@{^(->}[r]^\iota \ar@{^(->}[d] & \Dcal^\times \times \Dcal^\times \ar@{->}[d] \\
[\GO(V_{(1)})\times \GO(V_{(-\gamma)})] \ar@{^(->}[r] &  \GO(V_\Dcal). 
}$$
${}$

Suppose $\varphi = \varphi_1\oplus \varphi_2 \in S(V_\Dcal(k_\AA))$, where $\varphi_1 \in S(V_{(1)}(k_\AA))$ and $\varphi_2 \in S(V_{(-\gamma)}(k_\AA))$.
In Section~\ref{sec 2.1} we put
$$\theta_{\chi}^K(g;\varphi_1) =  \int_{K^1\backslash K_\AA^1}\theta^{V_{(1)}}(g, r h_g; \varphi_1) \chi(r h_g) dr, \quad \forall g \in \GL_2^{{}^{\!\!+_{\!\tiny K}}} (k_\AA).$$
The Siegel-Weil formula in Theorem~\ref{thm 3.1} says
$$E(g,\frac{1}{2}; \varphi_2) 
= \frac{1}{L(1,\varsigma_K)} \cdot 
\int_{K^1\backslash K_\AA^1} \theta^{V_{(-\gamma)}} (g, rh_g; \varphi_2) dr.$$
Note that the following lemma is straightforward.

\begin{lem}\label{lem 4.2}
Given $g \in \GL_2^{{}^{\!\!+_{\!\tiny K}}}(k_\AA)$ and $h_1, h_2 \in K_\AA^\times$ with $\det(g) = \N_{K/k}(h_1) = \N_{K/k}(h_2)$, one has
$$
\theta^{V_{(1)}}(g,h_1;\varphi_1) \cdot \theta^{V_{(\gamma)}}(g,h_2;\varphi_2) 
= \theta^{V_\Dcal}\big(g, [h_1 h', h'];\varphi_1 \oplus \varphi_2\big).
$$
Here $h' \in K_\AA^\times$ is chosen so that $h'/\overline{h'} = h_2/h_1$, and $[h_1h', h'] \in (\Dcal_\AA^\times \times \Dcal_\AA^\times)/k_\AA^\times $ is considered as an element in $\GO(V_\Dcal)(k_\AA)$.
\end{lem}

Applying the \lq\lq seesaw identity\rq\rq\ (cf.\ \cite{Kul}) with respect to the following diagram
\begin{eqnarray}\label{seesaw}
\xymatrix@C=2pc@R=2pc{
 \,\,\,\,\,\,  \quad \GL_2^{{}^{\!\!+_{\!\tiny K}}} \quad  \ar@{^{(}->}[d]_{\text{diagonal}}\ar@{-}[rrd] &  & \ [\GO(V_{(1)})\times\GO(V_{(-\gamma)})] \ar@{^{(}->}[d] \\
[\GL_2^{{}^{\!\!+_{\!\tiny K}}} \times \GL_2^{{}^{\!\!+_{\!\tiny K}}} ]  \ar@{-}[rru] & & \quad \,\,\,\,\,\,\,\, \GO(V_\Dcal)^{{}^{\!\!+_{\!\tiny K}}}, \quad
}
\end{eqnarray}
where $[\GL_2^{{}^{\!\!+_{\!\tiny K}}} \times \GL_2^{{}^{\!\!+_{\!\tiny K}}} ]$ (resp.\ $[\GO(V_{(1)})\times\GO(V_{(-\gamma)})]$) is the subgroup of $\GL_2^{{}^{\!\!+_{\!\tiny K}}} \times \GL_2^{{}^{\!\!+_{\!\tiny K}}}$ (resp.\ $\GO(V_{(1)})\times\GO(V_{(-\gamma)})$) consisting of all pairs $(g_1,g_2)$ where $g_1$ and $g_2$ have the same determinants (resp.\ the factor of similitudes),
we then obtain that:

\begin{prop}\label{prop 4.3}
Given $\phi \in \Pi$, and $\varphi \in S(V_\Dcal(k_\AA))$, we have
\begin{eqnarray}
\Zcal(\frac{1}{2}; \phi, \varphi)
&=& 
\frac{1}{L(1,\varsigma_K)} \cdot \int_{K^{\!^\times} k_\AA^\times \backslash K_\AA^\times} \int_{K^{\!^\times} k_\AA^\times \backslash K_\AA^\times} \theta^{\Dcal}( h_1,h_2;\phi, \varphi) \cdot \chi(h_1 h_2^{-1}) dh_1 dh_2. \nonumber
\end{eqnarray}
\end{prop}

\begin{proof}
The above discussion says that
\begin{eqnarray}
\Zcal(\frac{1}{2};\phi,\varphi) &=&  \frac{1}{L(1,\varsigma_K)} \cdot \int_{\Z(k_\AA) \GL_2^{{}^{\!\!+_{\!\tiny K}}}(k) \backslash \GL_2^{{}^{\!\!+_{\!\tiny K}}}(k_\AA)} \phi(g)  \nonumber \\
& &\cdot 
\left(\int_{K^1\backslash K_\AA^1}\int_{K^1\backslash K_\AA^1}
\theta^{V_{(1)}}(g,r_1 h_g;\varphi_1) \theta^{V_{(-\gamma)}}(g,r_2 h_g; \varphi_2) \chi(r_1 h_g) d r_1 dr_2\right) dg \nonumber \\
&=& \frac{1}{L(1,\varsigma_K)} \cdot
\int_{\SL_2(k)\backslash \SL_2(k_\AA)} \phi\big(g^1\alpha(\N_{K/k}(h))\big)  \nonumber \\
& & \cdot \left(\int_{K^{\!^\times} k_\AA^\times \backslash K_\AA^\times} \int_{K^{\!^\times} k_\AA^\times \backslash K_\AA^\times}
\theta^\Dcal(g^1\alpha(\N_{K/k}(h)), [h h',h']; \varphi) \chi(h) d h dh'\right) dg^1 \nonumber \\
&=& 
\frac{1}{L(1,\varsigma_K)} \cdot \int_{K^{\!^\times} k_\AA^\times \backslash K_\AA^\times} \int_{K^{\!^\times} k_\AA^\times \backslash K_\AA^\times} \theta^{\Dcal}( h_1,h_2; \phi,\varphi) \cdot \chi(h_1 h_2^{-1}) dh_1 dh_2. \nonumber
\end{eqnarray}
\end{proof}

For each pair $(\phi, \varphi)$ with $\phi \in \Pi$ and $\varphi \in S(V_\Dcal(k_\AA))$, define the global toric period integral by:
\begin{eqnarray}\label{eqn 4.2}
\Tcal(\phi,\varphi) &:=& \int_{K^{\!^\times} k_\AA^\times \backslash K_\AA^\times} \int_{K^{\!^\times} k_\AA^\times \backslash K_\AA^\times} \theta^{\Dcal}( h_1,h_2; \phi, \varphi) \cdot \chi(h_1 h_2^{-1}) dh_1 dh_2. 
\end{eqnarray}
Then by Corollary~\ref{cor 3.3}, Proposition~\ref{prop 4.1} and \ref{prop 4.3}, we arrive at:

\begin{cor}\label{cor 4.4}
Given pure tensors $\phi = \otimes_v \phi_v \in \Pi$ and $\varphi = \otimes_v \varphi_v \in S(V_\Dcal(k_\AA))$, we have
$$\Tcal(\phi,\varphi) = L(\frac{1}{2},\Pi \times \chi)\cdot \prod_v \Tcal_v(\phi_v, \varphi_v).$$
\end{cor}




\section{Waldspurger formula}\label{sec 5}

Let $\Pi$ be an automorphic cuspidal representation of $\GL_2(k_\AA)$ with a unitary central character $\eta$. For a quaternion algebra $\Dcal$ over $k$,
let $\Pi^{\Dcal}$ be, if exists, the automorphic representation of $\Dcal_\AA^\times$ corresponding to $\Pi$ via the Jacquet-Langlands correspondence.
Let $K$ be a separable quadratic algebra over $k$ together with an embedding $\iota: K \hookrightarrow \Dcal$.
Given a unitary Hecke character $\chi: K^\times \backslash K_\AA^\times \rightarrow \CC^\times$, suppose $ \eta\cdot \chi\big|_{k_\AA^\times} = 1$.
For each $f \in \Pi^\Dcal$, put
$$P_\chi^\Dcal(f):= \int_{K^{\!^\times} k_\AA^\times \backslash K_\AA^\times} f\big(\iota(h)\big) \chi(h) d^\times h.$$
This induces a linear functional $\Pcal_\chi^\Dcal : \Pi^\Dcal \otimes \widetilde{\Pi}^\Dcal \rightarrow \CC$ defined by
$$\Pcal_\chi^\Dcal(f\otimes \tilde{f}) := P_\chi^\Dcal(f) \cdot P_{\chi^{-1}}^\Dcal(\tilde{f}), \quad \forall f\otimes\tilde{f} \in \Pi^\Dcal \otimes \widetilde{\Pi}^\Dcal.$$

On the other hand, write $\Pi^\Dcal = \otimes_v \Pi_v^\Dcal$ and $\widetilde{\Pi}^\Dcal = \otimes_v \widetilde{\Pi}_v^\Dcal$. For each place $v$ of $k$, let $\langle \cdot,\cdot \rangle_v : \Pi_v^\Dcal \times \widetilde{\Pi}_v^\Dcal \rightarrow \CC$ be the natural duality pairing. We assume that the identification between $\Pi^\Dcal$ (resp.\ $\widetilde{\Pi}^\Dcal$)
and $\otimes_v \Pi_v^\Dcal$ (resp.\ $\otimes_v \widetilde{\Pi}_v^\Dcal$) satisfies:
\begin{eqnarray} \label{eqn 5.1}
\langle \cdot , \cdot \rangle_{\text{Pet}} = \frac{2 L(1,\Pi,\text{Ad})}{\zeta_k(2)} \cdot \prod_v \langle \cdot,\cdot \rangle_v,
\end{eqnarray}
where $\langle \cdot , \cdot \rangle_{\text{Pet}} : \Pi^\Dcal \times \widetilde{\Pi}^\Dcal \rightarrow \CC$ is the pairing induced from the Petersson inner product on $\Pi^\Dcal$.
The local toric period integral $\Pscr_{\chi,v}^\Dcal  :\Pi_v^\Dcal \otimes \widetilde{\Pi}_v^\Dcal \rightarrow \CC$ is defined by:
$$\Pscr_{\chi,v}^\Dcal(f_v\otimes \tilde{f}_{v} ):= \frac{L_v(1,\varsigma_{K}) L_v(1,\Pi, \text{Ad})}{L_v(1/2,\Pi \times \chi) \zeta_v(2)} \cdot 
\int_{K_v^\times/ k_v^\times}  \langle \Pi^\Dcal_v\big(\iota(h_v)\big) f_{v}, \tilde{f}_v \rangle_{v} \cdot \chi_v(h_v) d^\times h_v.$$

\begin{lem}\label{lem 5.1}
Suppose $v$ is \lq\lq good,\rq\rq\ i.e.\ 
the additive character $\psi_v$ has trivial conductor, the quaternion algebra $\Dcal$ splits at $v$, the local representation $\Pi_v^\Dcal = \Pi_v$ is an unramified principal series, the place $v$ is unramified in $K$, the character $\chi_v$ is unramified.
Take $f_v \in \Pi_v^\Dcal$ and $\tilde{f}_v \in \widetilde{\Pi}_v^\Dcal$ to be spherical and invariant by $\iota(O_{K_v})$ with $\langle f_v, \tilde{f}_v\rangle_v = 1$. 
Then $$\Pscr_{\chi,v}^\Dcal(f_v\otimes \tilde{f}_{v} ) = 1.$$
\end{lem} 
\begin{proof}
Suppose $v$ is inert in $K$. Then the choices of $f_v$ and $\tilde{f}_v$ satisfy
$$\Pscr_{\chi,v}^\Dcal(f_v\otimes \tilde{f}_{v} ) = \frac{L_v(1,\varsigma_{K}) L_v(1,\Pi, \text{Ad})}{L_v(1/2,\Pi \times \chi) \zeta_v(2)}.$$
It is straightforward that the right hand side of the above equality equals to $1$ under the above assumptions on $v$.\\ 

Suppose $v$ splits in $K$, i.e.\ $K_v = k_v \times k_v$.
Write $\chi_v = \chi_{v,1} \times \chi_{v,2}$ on $k_v^\times \times k_v^\times$.
Then
$$L_v(\frac{1}{2} ,\Pi \times \chi) = L_v(\frac{1}{2}, \Pi_v \otimes \chi_{v,1}) \cdot L_v(\frac{1}{2},\Pi_v \otimes \chi_{v,2}) = 
L_v(\frac{1}{2}, \Pi_v \otimes \chi_{v,1}) \cdot L_v(\frac{1}{2},\widetilde{\Pi}_v \otimes \chi_{v,1}^{-1}).$$
The last equality follows from the assumption $\eta_v \cdot \chi_v \big|_{k_v^\times} = 1$, where $\eta_v$ is the central character of $\Pi_v$.
The pairing $\langle \cdot , \cdot \rangle_v$ can be realized by
$$\langle f_v,\tilde{f}_v \rangle_v:= \frac{\zeta_v(2)}{\zeta_v(1) L_v(1,\Pi,\text{Ad})} \cdot \int_{k_v}W_{f_v}\begin{pmatrix} a_v & 0 \\ 0 & 1\end{pmatrix} W_{\tilde{f}_v}'\begin{pmatrix} a_v & 0 \\ 0 & 1\end{pmatrix} d^\times a_v, \forall f_v \in \Pi_v^\Dcal,\ \tilde{f}_v \in \Pi_v^\Dcal,$$
where $W_{f_v}$ (resp.\ $W_{\tilde{f}_v}'$) is the Whittaker function of $f_v$ (resp.\ $\tilde{f}_v$) with respect to $\psi_v$ (resp.\ $\overline{\psi}_v$).
We may assume the embedding $\iota: K_v \rightarrow \Dcal_v$ satisfies
$$\iota(a_v,a_v') = \begin{pmatrix} a_v&0\\ 0 & a_v'\end{pmatrix} \in \Mat_2(k_v) = \Dcal_v, \quad \forall (a_v,a_v') \in k_v\times k_v.$$
Then
\begin{eqnarray}
\Pscr_{\chi,v}^\Dcal(f_v\otimes \tilde{f}_{v} )
&=& \left(\frac{1}{L_v(1/2,\Pi_v\otimes \chi_{v,1})} \int_{k_v^\times} W_{f_v}\begin{pmatrix} a_v & 0 \\ 0 & 1\end{pmatrix} \chi_{v,1}(a_v)d^\times a_v\right) \nonumber \\
&& \cdot 
\left(\frac{1}{L_v(1/2,\widetilde{\Pi}_v\otimes \chi_{v,1}^{-1})} \int_{k_v^\times} W_{\tilde{f}_v}'\begin{pmatrix} a_v & 0 \\ 0 & 1\end{pmatrix} \chi_{v,1}^{-1}(a_v)d^\times a_v\right). \nonumber
\end{eqnarray}
Therefore when $f_v$ and $\tilde{f}_v$ are spherical and invariant by $\iota(O_{K_v})$ with $\langle f_v, \tilde{f}_v\rangle_v = 1$, we get $\Pscr_{\chi,v}^\Dcal(f_v\otimes \tilde{f}_{v} ) = 1$.
\end{proof}

Set $\Pscr_\chi^\Dcal := \otimes_v \Pscr_{\chi,v}^\Dcal : \Pi^\Dcal \otimes \widetilde{\Pi}^\Dcal \rightarrow \CC$.
We finally arrive at:

\begin{thm}\label{thm 5.1}
The linear functionals $\Pcal_\chi^\Dcal$ and $\Pscr_\chi^\Dcal$ on $\Pi^\Dcal \otimes \widetilde{\Pi}^\Dcal$ satisfy
$$ \Pcal_\chi^\Dcal = L(\frac{1}{2},\Pi \times \chi) \cdot \Pscr_\chi^\Dcal.$$
\end{thm}

\begin{proof}
The case when $K = k \times k$ is proven in Appendix~\ref{sec C}. Suppose $K$ is a quadratic field over $k$.
Take pure tensors $\phi = \otimes_v \phi_v \in \Pi$ and $\varphi = \otimes_v \varphi_v \in S(V_\Dcal(k_\AA))$. Applying the global and local Shimizu correspondence (cf.\ Theorem~\ref{thm 2.2} and Section~\ref{sec 2.3.1}), Corollary~\ref{cor 4.4} implies
\begin{eqnarray}
\Pcal_\chi^\Dcal\big(\textbf{Sh}(\theta^\Dcal(\cdot,\cdot;\phi,\varphi))\big) &=& \Tcal(\phi,\varphi) \nonumber \\
&=& L(\frac{1}{2},\Pi\times \chi) \cdot \prod_v \Tcal_v(\phi_v,\varphi_v) \nonumber \\
&=& L(\frac{1}{2},\Pi\times \chi) \cdot \prod_v \Pscr_{\chi,v}^\Dcal\big(\textbf{Sh}_v(\theta^{\Dcal,o}(\cdot,\cdot;\phi_v,\varphi_v))\big) \nonumber \\
&=& L(\frac{1}{2},\Pi\times \chi) \cdot \Pscr_{\chi}^\Dcal\big(\theta^\Dcal(\cdot,\cdot;\phi, \varphi)\big). \nonumber
\end{eqnarray}
Therefore the result holds.
\end{proof}

\subsection{Non-vanishing criterion}\label{sec 5.1}

For each place $v$ of $k$, we have:

\begin{lem}\label{lem 5.2}
\text{\rm (cf.\ \cite{Tun})} The space $\Hom_{K_v^\times}( \Pi_v^\Dcal, \chi_v^{-1})$ is at most one dimensional. Moreover, $\Hom_{K_v^\times}( \Pi_v^\Dcal, \chi_v^{-1}) \neq 0$ if and only if
\begin{eqnarray}\label{eqn 5.2}
\epsilon_v(\Pi_v \times \chi_v) = \chi_v(-1) \varsigma_{K,v}(-1) \cdot \epsilon_v(\Dcal).
\end{eqnarray}
Here $\epsilon_v(\Pi_v \times \chi_v)$ is the local root number of $L_v(s,\Pi \times \chi)$, and $\epsilon_v(\Dcal)$ is the Hasse invariant of $\Dcal$ at $v$.
\end{lem}

It is clear that $\Pscr_{\chi,v}^\Dcal$ lies in $\Hom_{K_v^\times}( \Pi_v^\Dcal, \chi_v^{-1}) \otimes \Hom_{K_v^\times}(\widetilde{\Pi}_v^\Dcal, \chi_v)$.
Moreover, following Waldspurger~\cite[Lemme 10]{Wa3} one gets
\begin{lem}\label{lem 5.3}
$\Pscr_{\chi,v}^\Dcal$ is a generator of the $\CC$-vector space $\Hom_{K_v^\times}( \Pi_v^\Dcal, \chi_v^{-1}) \otimes \Hom_{K_v^\times}(\widetilde{\Pi}_v^\Dcal, \chi_v)$.
\end{lem}

Consequently, $\Pscr_\chi^\Dcal$ generates the space $\Hom_{K_\AA^\times}( \Pi^\Dcal, \chi^{-1}) \otimes \Hom_{K_\AA^\times}(\widetilde{\Pi}^\Dcal, \chi)$, in which $\Pcal_\chi^\Dcal$ lies. Therefore Theorem~\ref{thm 5.1} implies:

\begin{cor}\label{cor 5.4}
Let $\Pi$ be an automorphic cuspidal representation of $\GL_2(k_\AA)$ with a unitary central character $\eta$. Given a separable quadratic algebra $K$ over $k$ and a unitary Hecke character $\chi : K^\times \backslash K_\AA^\times \rightarrow \CC^\times$ with $\eta \cdot \chi\big|_{k_\AA^\times} = 1$, assume $\prod_v \epsilon_v(\Pi_v\times \chi_v) = 1$.
Let $\Dcal$ be the quaternion algebra over $k$ satisfying \text{\rm (\ref{eqn 5.2})} for every place $v$ of $k$, and $\Pi^\Dcal$ be the automorphic representation of $\Dcal_\AA^\times$ corresponding to $\Pi$ via the Jacquet-Langlands correspondence. Choose an embedding $\iota: K \hookrightarrow \Dcal$.
Then the non-vanishing of the central critical value $L(1/2,\Pi \times \chi)$ is equivalent to the existence of $f \in \Pi^\Dcal$ so that
$$ P_\chi^\Dcal (f) = \int_{K^{\!^\times} k_\AA^\times \backslash K_\AA^\times} f\big(\iota(h)\big) \chi(h) d^\times h \neq 0.$$
\end{cor}

\appendix{}

\section{Local Shimizu correspondences}\label{sec A}

Recall the Shimizu correspondence stated in Theorem~\ref{thm 2.2}:

\begin{thm}\label{thm A.1}
Given an infinite dimensional automorphic representation $\Pi^\Dcal$ of $\Dcal_\AA^\times$ which is cuspidal if $\Dcal = \Mat_2$, suppose the central character of $\Pi^\Dcal$ is unitary. Then 
$$
\Theta^\Dcal(\Pi)= \{ f_1 \otimes \bar{f}_2 : \Dcal_\AA^\times \times \Dcal_\AA^\times \rightarrow \CC \mid f_1, f_2 \in \Pi^\Dcal \}_{\CC-\text{\rm span}}.
$$
Here $f_1\otimes \bar{f}_2$ is viewed as the function $\big((b,b') \mapsto f_1(b) \cdot \bar{f}_2(b')\big)$.
Consequently, let $\widetilde{\Pi}^\Dcal$ be the contragredient representation of $\Pi^\Dcal$. Identifying $\widetilde{\Pi}^\Dcal$ with the space $\{ \bar{f} \mid f \in \Pi^\Dcal\}$ via the Petersson inner product, the equality~\text{\rm (\ref{eqn Sh})} induces an isomorphism
$$
\text{\bf Sh}: \Theta^\Dcal(\Pi) \cong \Pi^\Dcal \otimes \widetilde{\Pi}^\Dcal.
$$
\end{thm}

Recall that for $\phi_v \in \Pi_v$ and $\varphi_v \in S(V_\Dcal(k_v))$, 
we define
$\theta_v^{\Dcal,o} (b_v,b_v';\phi_v,\varphi_v)$ for $b_v,b_v' \in \Dcal_v^\times$ in Section~\ref{sec 2.3.1} by
\begin{eqnarray}
&& \theta_v^{\Dcal,o} (b_v,b_v';\phi_v,\varphi_v) \nonumber \\
&=& \frac{\zeta_v(2)}{L_v(1,\Pi,\Ad)} \cdot
\int_{\U(k_v) \backslash \SL_2(k_v)} W_{\phi_v}\big(g_v^1 \alpha(b_v b_v'^{-1})\big) \cdot 
\big(\omega^\Dcal_v(g_v^1 \alpha(b_v b_v'^{-1}), [b_v,b_v'])\varphi_v\big)(1) d g_v^1. \nonumber
\end{eqnarray}

\begin{lem}
Suppose $v$ is \lq\lq good,\rq\rq\
i.e.\ $\psi_v$ has trivial conductor, the representation $\Pi_v$ is an unramified principle series, the vector $\phi_v \in \Pi_v$ is spherical with $W_{\phi_v}\begin{pmatrix} 1& 0 \\ 0 &1 \end{pmatrix} = 1$, the quaternion algebra $\Dcal_v = \Mat_2(k_v)$, and the Schwartz function $\varphi_v = \mathbf{1}_{\Mat_2(O_v)}$.
One has
$$\theta_v^{\Dcal,o}(b_v,b_v';\phi_v,\varphi_v) = 1, \quad \forall b_v,b_v' \in \GL_2(O_v).$$
\end{lem}

\begin{proof}
From the Iwasawa decomposition $\SL_2(k_v) = B^1(k_v) \cdot \SL_2(O_v)$, the above assumptions imply that for $b_v,b_v' \in \GL_2(O_v)$, we have
\begin{eqnarray}
\theta_v^{\Dcal,o}(b_v,b_v';\phi_v,\varphi_v) &=&
\frac{\zeta_v(2)}{L_v(1,\Pi,\Ad)} \cdot \int_{k_v^\times} W_{\phi_v}\begin{pmatrix} a_v & 0 \\ 0 & a_v^{-1}\end{pmatrix} \mathbf{1}_{O_v}(a_v) d^\times a_v \nonumber \\
&=& 1\nonumber
\end{eqnarray}
\end{proof}

The aim of this section is to show:

\begin{thm}\label{thm A.2}
Given pure tensors $\phi = \otimes_v \phi_v \in \Pi$ and $\varphi = \otimes_v \varphi_v \in S(V_\Dcal(k_\AA))$,
we have
$$\int_{\Dcal^\times k_\AA^\times \backslash \Dcal_\AA^\times} \theta^\Dcal(bb_1,bb_2;\phi,\varphi) d^\times b
= \frac{2 L(1,\Pi,\text{\rm Ad})}{\zeta_k(2)} \cdot \prod_v \theta_v^{\Dcal,o}(b_{1,v}, b_{2,v}; \phi_v,\varphi_v), \quad \forall b_1,b_2 \in \Dcal_\AA^\times.$$
\end{thm}

The proof of the above theorem is given in Section~\ref{sec A.1} when $\Dcal = \Mat_2$, and in Section~\ref{sec A.2} when $\Dcal$ is division.\\

Via the Petersson pairing $\langle \cdot, \cdot \rangle_{\text{Pet}}^\Dcal: \Pi^\Dcal \times \widetilde{\Pi}^\Dcal$, the representation $\Pi^\Dcal \otimes \widetilde{\Pi}^\Dcal$ is isomorphic to the space of the \textit{matrix coefficients} of $\Pi^\Dcal \otimes \widetilde{\Pi}^\Dcal$:
$$f\otimes \tilde{f} \longleftrightarrow m_{f\otimes \tilde{f}} \quad \forall f\otimes \tilde{f} \in \Pi^\Dcal \times \widetilde{\Pi}^\Dcal,$$
where
$$m_{f\otimes \tilde{f}}(b,b') := \langle \Pi^\Dcal(b) f, \widetilde{\Pi}^\Dcal(b') \tilde{f}\rangle_{\text{Pet}}, \quad \forall b,b' \in \AA_\Dcal^\times.$$
On the other hand,
for each place $v$ of $k$, we may also identify $\Pi^\Dcal_v \otimes \widetilde{\Pi}^\Dcal_v$ with the space of matrix coefficients, i.e.\ for $f_v \in \Pi^\Dcal_v$ and $\tilde{f}_v \in \widetilde{\Pi}_v^\Dcal$, the matrix coefficient $m_{f_v \otimes \tilde{f}_v}$ associated to $m_{f_v \otimes \tilde{f}_v}$ is defined by
$$m_{f_v \otimes \tilde{f}_v}(b_v, b_v') := \langle \Pi_v^\Dcal(b_v) f_v, \widetilde{\Pi}_v^\Dcal(b_v') \tilde{f}_v\rangle_v^\Dcal, \quad \forall b_v, b_v' \in \Dcal_v^\times.$$
Here $\langle \cdot,\cdot\rangle_v^\Dcal : \Pi_v^\Dcal \times \widetilde{\Pi}_v^\Dcal \rightarrow \CC$ is the natural duality pairing.
Put $$\Theta_v^\Dcal(\Pi_v):= \{ \theta_v^{\Dcal,o}(\cdot,\cdot;\phi_v,\varphi_v)\mid \phi_v \in \Pi_v,\ \varphi_v \in S(V_\Dcal(k_v))\}.$$

\begin{prop}\label{prop A.L}
We have the following equality:
$$\Theta_v^\Dcal(\Pi_v) = \{ m_{f_v \otimes \tilde{f}_v} \mid f_v \in \Pi_v^\Dcal,\ \tilde{f}_v \in \widetilde{\Pi}_v^\Dcal\}_{\CC-\text{\rm span}}.$$
This induces an isomorphism $\text{\bf Sh}_v : \Theta_v^\Dcal(\Pi_v) \cong \Pi_v^\Dcal \otimes \widetilde{\Pi}_v^\Dcal$.
\end{prop}

\begin{proof}
Pick $\phi^o = \otimes_v \phi^o_{v} \in \Pi $ and $\varphi^o = \otimes_v \varphi^o_{v} \in S(V_\Dcal(k_\AA))$ so that
$$ C(\phi^o,\varphi^o):= \int_{\Dcal^\times k_\AA^\times \backslash \Dcal_\AA^\times} \theta^{\Dcal}(b,b;\phi^o,\varphi^o) db \neq 0.$$ 
Then for each place $v_0$ of $k$,
the space of matrix coefficients of $\Pi_{v_0}^\Dcal$ can be generated by $m_{v_0}(\phi_{v_0},\varphi_{v_0})$ for $\phi_{v_0} \in \Pi_{v_0}$ and $\varphi_{v_0} \in S(V_\Dcal(k_v))$, where $m_v(\phi_{v_0},\varphi_{v_0})$ is defined by:
$$m_v(\phi_{v_0},\varphi_{v_0})(b_{v_0},b_{v_0}') := \int_{\Dcal^\times k_\AA^\times \backslash \Dcal_\AA^\times} \theta^\Dcal(b b_{v_0}, bb_{v_0}';\phi, \varphi) db, \quad \forall b_{v_0},b_{v_0}' \in \Dcal_{v_0}^\times, $$
where $\phi = \phi_{v_0} \otimes_{v\neq v_0} \phi^o_{v} \in \Pi$ and $\varphi = \varphi_{v_0} \otimes_{v \neq v_0} \varphi^o_{v} \in S(V_\Dcal(k_\AA))$.
By Theorem~\ref{thm A.2} we may assume that the chosen $\phi^o$ and $\varphi^o$ satisfy
$$C(\phi^o,\varphi^o) = \theta_{v_0}^{\Dcal,o}(1,1;\phi^o_{v_0},\varphi^o_{v_0}).$$
Then
\begin{eqnarray}
m_{v_0}(\phi_{v_0},\varphi_{v_0})(b_{v_0},b_{v_0}')
&=& \frac{ \int_{\Dcal^\times k_\AA^\times \backslash \Dcal_\AA^\times} \theta^\Dcal(b b_{v_0}, bb_{v_0}';\phi, \varphi) db}{C(\phi^o,\varphi^o)} \cdot C(\phi^o,\varphi^o) \nonumber \\
&=&\frac{\theta_{v_0}^{\Dcal,o}(b_{v_0},b_{v_0}'; \phi_{v_0},\varphi_{v_0})}{\theta_{v_0}^{\Dcal,o}(1,1; \phi^o_{v_0},\varphi^o_{v_0})} \cdot C(\phi^o,\varphi^o) \quad \quad \text{(by Theorem~\ref{thm A.2})}
 \nonumber \\
&=& \theta_{v_0}^{\Dcal,o}(b_{v_0},b_{v_0}'; \phi_{v_0},\varphi_{v_0}).\nonumber
\end{eqnarray}

\end{proof}

\subsection{Proof of Theorem~\ref{thm A.2} when $\Dcal = \Mat_2$} \label{sec A.1}

Given $\phi \in \Pi$ and $\varphi \in S(V_\Dcal(k_\AA))$, consider the Whittaker function associated to $\phi$ and $\varphi$:
\begin{eqnarray}
&&\Wcal_{\phi,\varphi}(b_1,b_2) \nonumber \\
&:=&
\int_{k\backslash k_\AA} \int_{k\backslash k_\AA} \theta^\Dcal\left(\begin{pmatrix} 1 & u_1 \\ 0 & 1 \end{pmatrix} b_1, 
\begin{pmatrix} 1 & u_2 \\ 0 & 1 \end{pmatrix} b_2; \phi, \varphi\right) \psi(u_2-u_1) du_1du_2, \quad \forall b_1,b_2 \in \GL_2(k_\AA).
\nonumber
\end{eqnarray}
Then:

\begin{prop}\label{prop A.5}
When $\phi$ and $\varphi$ are both pure tensors, we have
$$\Wcal_{\phi,\varphi} = \prod_v \Wcal_{\phi,\varphi,v}.$$
Here for $b_1,b_2 \in \GL_2(k_v)$, let
$$\Wcal_{\phi,\varphi,v}(b_1,b_2):= \int_{\U(k_v)\backslash \SL_2(k_v)} W_{\phi_v}(g^1 \alpha(b_1b_2^{-1})) \cdot 
\Big(\omega_v^\Dcal(g^1 \alpha(b_1b_2^{-1}),[b_1,b_2]) \varphi_v\Big)^\sim\begin{pmatrix}1&1\\0&1\end{pmatrix} d g^1,$$
and
$$(\varphi_v)^\sim\begin{pmatrix} a&b\\c&d\end{pmatrix} := \int_{k_v} \varphi_v\begin{pmatrix} a&b'\\c&d\end{pmatrix} \psi_v(bb') db' \quad \text{ for $\varphi_v \in S(V_\Dcal(k_v))$.}$$
\end{prop}

\begin{proof}
Let $V_1:=\left\{\begin{pmatrix} * & 0 \\ 0&*\end{pmatrix}\right\}\subset V_\Dcal$, $V_2 := \left\{\begin{pmatrix}0&*\\ *&0\end{pmatrix}\right\} \subset V_\Dcal$, and $Q_i:= Q_\Dcal\big|_{V_i}$ for $i = 1,2$.
Then $(V_\Dcal,Q_\Dcal) = (V_1,Q_1)\oplus (V_2,Q_2)$.
For $\varphi_2 \in S(V_2(k_\AA))$ and $g^1 \in \SL_2(k_\AA)$, it is observed that
$$\big(\omega^{V_2}(g^1)\varphi_2 \big)^\sim \begin{pmatrix} 0 & b \\ c&0 \end{pmatrix} = \varphi_2^\sim\begin{pmatrix} 0 & b' \\ c' &0 \end{pmatrix},$$
where $(c',b') = (c,b) \cdot g^1$.
Thus for $\varphi \in S(V_\Dcal(k_\AA))$, by Poisson summation formula we may write 
$$\theta^{V_\Dcal} (g^1\alpha(b_1b_2^{-1}),[b_1,b_2];\varphi) = \theta^{V_\Dcal}_1(g^1\alpha(b_1b_2^{-1}),[b_1,b_2];\varphi) +
\theta^{V_\Dcal} _2(g^1\alpha(b_1b_2^{-1}),[b_1,b_2];\varphi),$$
where
$$\theta_1^{V_\Dcal} (g^1\alpha(b_1b_2^{-1}),[b_1,b_2];\varphi) := \sum_{\gamma \in \U(k)\backslash \SL_2 (k)} \sum_{a,d \in k} 
\big(\omega^\Dcal(\gamma g^1\alpha(b_1b_2^{-1}),[b_1,b_2])\varphi\big)^\sim\begin{pmatrix} a&1\\0&d\end{pmatrix},$$
and
$$\theta_2^{V_\Dcal} (g^1\alpha(b_1b_2^{-1}),[b_1,b_2];\varphi) :=
\sum_{a,d \in k} 
\big(\omega^\Dcal(g^1\alpha(b_1b_2^{-1}),[b_1,b_2])\varphi\big)^\sim\begin{pmatrix} a&0\\0&d\end{pmatrix}.$$
Since for $u_1,u_2 \in k_\AA$ one has
$$\Big(\omega^\Dcal(1,\left[\begin{pmatrix}1&u_1\\0&1\end{pmatrix},\begin{pmatrix}1&u_2\\0&1\end{pmatrix}\right])\varphi\Big)^\sim\begin{pmatrix}a&b\\c&d\end{pmatrix}
= \psi\big(b(-au_2+du_1+cu_1u_2)\big) \cdot \varphi^\sim\begin{pmatrix}a&b\\c&d\end{pmatrix},$$
we get
\begin{eqnarray}
&&\int_{k\backslash k_\AA} \int_{k\backslash k_\AA} \theta^{V_\Dcal} \left(g^1\alpha(b_1b_2^{-1}), \left[\begin{pmatrix} 1 & u_1 \\ 0 & 1 \end{pmatrix} b_1, 
\begin{pmatrix} 1 & u_2 \\ 0 & 1 \end{pmatrix} b_2\right]; \varphi\right) \psi(u_2-u_1) du_1du_2 \nonumber \\
&=&
\sum_{\gamma \in \U(k)\backslash \SL_2 (k)}
\big(\omega^\Dcal(\gamma g^1\alpha(b_1b_2^{-1}),[b_1,b_2])\varphi\big)^\sim\begin{pmatrix} 1&1\\0&1\end{pmatrix}. \nonumber
\end{eqnarray}
Therefore
\begin{eqnarray}
\Wcal_{\phi,\varphi}(b_1,b_2) &=& 
\int_{\U(k)\backslash \SL_2(k_\AA)} \phi(g^1 \alpha(b_1b_2^{-1})) \cdot \big(\omega^\Dcal(\gamma g^1\alpha(b_1b_2^{-1}),[b_1,b_2])\varphi\big)^\sim\begin{pmatrix} 1&1\\0&1\end{pmatrix} dg^1 \nonumber \\
&=& \int_{\U(k_\AA) \backslash \SL_2(k_\AA)} W_\phi(g^1\alpha(b_1b_2^{-1}))
\cdot \big(\omega^\Dcal(\gamma g^1\alpha(b_1b_2^{-1}),[b_1,b_2])\varphi\big)^\sim\begin{pmatrix} 1&1\\0&1\end{pmatrix} dg^1 \nonumber \\
&=& \prod_v \Wcal_{\phi,\varphi,v}(b_{1,v},b_{2,v}).\nonumber
\end{eqnarray}
\end{proof}

The space consisting of all the $\Wcal_{f,\varphi,v}$ is actually the Whittaker model of $\Pi_v \otimes \widetilde{\Pi}_v$. Moreover,
the following straightforward lemma connects the local Whittaker function $\Wcal_{\phi,\varphi,v}$ with $\theta_v^{\Dcal,o}(\cdot,\cdot;\phi_v,\varphi_v)$:

\begin{lem}\label{lem A.6}
For $b_1,b_2 \in \GL_2(k_v)$, one gets
$$\frac{\zeta_v(2)}{L_v(1,\Pi,\text{\rm Ad})} \cdot \int_{k_v^\times} \Wcal_{f,\varphi,v}\left(\begin{pmatrix} a_v&0\\ 0 &1\end{pmatrix} b_1,\begin{pmatrix} a_v&0\\ 0 &1\end{pmatrix} b_2\right) d^\times a_v = \zeta_v(1) \cdot \theta_v^{\Dcal,o}(b_1,b_2; \phi_v,\varphi_v).$$
\end{lem}

Recall that for pure tensors $f_1,f_2 \in \Pi^\Dcal = \Pi$, we have
$$\langle f_1,f_2\rangle_{\text{Pet}}^{\Mat_2} = \frac{2 \cdot L(1,\Pi,\text{Ad})}{\zeta_k(2)} \cdot \langle f_{1,v},\bar{f}_{2,v}\rangle_v^{\Mat_2},$$
where
$$\langle f_{1,v},\bar{f}_{2,v}\rangle_v^{\Mat_2}:= \frac{\zeta_v(2)}{\zeta_v(1)L_v(1,\Pi,\text{Ad})} \cdot \int_{k_v^\times} W_{f_{1,v}}\begin{pmatrix}a_v&0\\0&1\end{pmatrix} \overline{W_{f_{2,v}}\begin{pmatrix}a_v&0\\0&1\end{pmatrix}} d^\times a_v,$$
and $\langle f_{1,v},\bar{f}_{2,v}\rangle_v = 1$ when $v$ is \lq\lq good.\rq\rq\
Therefore by Proposition~\ref{prop A.5} and Lemma~\ref{lem A.6}, the Shimizu correspondence in Theorem~\ref{thm 2.2} implies that:

\begin{prop} \label{prop A.7}
\text{\rm Theorem~\ref{thm A.2}} holds when $\Dcal = \Mat_2$.
\end{prop}

\subsection{Proof of Theorem~\ref{thm A.2} when $\Dcal$ is division}\label{sec A.2}

Given $g^1 \in \SL_2(k_\AA)$,
we set
$$I^\Dcal(g^1,s,\varphi):= \sum_{\gamma \in \B^1(k)\backslash \SL_2(k)} |a(\gamma g^1)|_{\AA}^{s-1} \cdot \sum_{x \in k} \big(\omega^{\Dcal}(\gamma g^1)\varphi\big) (x), \quad \forall \varphi \in S(V_\Dcal(k_\AA)).$$
Here $a(g^1) = a \in k_\AA^\times$ is chosen so that $g^1$ can be written as
$$g^1 = \begin{pmatrix} a & * \\ 0 & a^{-1} \end{pmatrix} \kappa^1 \quad \text{ with } \kappa^1 \in \SL_2(O_{\AA}).$$
This series converges absolutely when $\re(s) > 3/2$, and has meromorphic continuation to the whole complex $s$-plane. 
Given pure tensors $\phi = \otimes_v \phi_v \in \Pi$ and $\varphi = \otimes_v \varphi_v \in S(V_\Dcal(k_\AA))$, one has
\begin{eqnarray}
J^\Dcal(s;\phi,\varphi)&:=& \int_{\SL_2(k)\backslash \SL_2(k_\AA)} \phi(g^1) I^\Dcal(g^1,s,\varphi) dg^1 \nonumber \\
&=& \int_{\U(k) \backslash \SL_2(k_\AA)} \phi(g^1) |a(g^1)|_{\AA}^{s-1} \cdot 
\big(\omega^{\Dcal}(g^1)\varphi\big) (1) d g^1 \nonumber \\
&=& \prod_v J_v^\Dcal(s;\phi_v,\varphi_v),\nonumber
\end{eqnarray}
where
$$J_v^\Dcal(s;\phi_v,\varphi_v):= \int_{\U(k_v)\backslash \SL_2(k_v)}W_{\phi_v}(g^1) \big(\omega_v^\Dcal(g^1)\varphi_v\big)(1)|a(g^1)|_v^{s-1} d g^1.$$
It is clear that:

\begin{lem}\label{lem A.8}
The local integral $J_v^\Dcal(s;\phi_v,\varphi_v)$ always converges when $\re(s) \geq 1$. Moreover, when $v$ is \lq\lq good,\rq\rq\ one has
$$J_v^\Dcal(s;\phi_v,\varphi_v) = \frac{L_v(s, \Pi, \text{\rm Ad})}{\zeta_v(2s)}.$$
In particular, we obtain that
$$J^\Dcal(s;\phi,\varphi) = \frac{L(s,\Pi,\text{\rm Ad})}{\zeta_k(2s)} \cdot \prod_v J_v^{\Dcal,o}(s;\phi_v,\varphi_v),$$
where
$$J_v^{\Dcal,o}(s;\phi_v,\varphi_v) := \frac{\zeta_v(2s)}{L_v(s, \Pi, \text{\rm Ad})} \cdot J_v^\Dcal(s;\phi_v,\varphi_v).$$
\end{lem}

For $b_1,b_2 \in \Dcal_v^\times$, one has
$$\theta_v^{\Dcal,o}(b_1,b_2; \phi_v,\varphi_v) = J_v^{\Dcal,o}(1; \phi_v', \varphi_v'),$$
where $\phi_v':= \Pi_v(\alpha(b_1b_2^{-1}))\phi_v$ and $\varphi_v' := \omega_v^\Dcal(\alpha(b_1b_2^{-1}), [b_1,b_2])\varphi_v$.
Thus we obtain that:

\begin{prop}\label{prop A.9}
{\rm Theorem~\ref{thm A.2}} holds when $\Dcal$ is division.
\end{prop}

\begin{proof}
From the equality~(\ref{eqn B.1}) in \textit{Remark}~\ref{rem B.2}, we get that for $b_1,b_2 \in \Dcal_\AA^\times$,
\begin{eqnarray}
\int_{\Dcal^\times k_\AA^\times \backslash \Dcal_\AA^\times} \theta^\Dcal(bb_1,bb_2;\phi,\varphi) db
&=& 2 J^\Dcal(1;\phi',\varphi') \nonumber \\
&=&  \frac{2 L(1,\Pi,\text{\rm Ad})}{\zeta_k(2)} \cdot \prod_v J^{\Dcal,o}(1;\phi'_v,\varphi_v') \nonumber \\
&=& \frac{2 L(1,\Pi,\text{\rm Ad})}{\zeta_k(2)} \cdot \prod_v \theta_v^{\Dcal,o}(b_{1,v}, b_{2,v}; \phi_v,\varphi_v). \nonumber
\end{eqnarray}

\end{proof}

Note that the equality~(\ref{eqn B.1}) follows from a \lq\lq metaplectic type\rq\rq\ Siegel-Weil formula, which is verified in the next section.

\section{Siegel-Weil formula for the metaplectic Eisenstein series}\label{sec B}

\subsection{Metaplectic groups}\label{sec B.1}
For each place $v$, the \it Kubota $2$-cocycle \rm $\sigma_v'$ 
is defined by (cf.\ \cite[Section 3]{Ku2}):
$$\sigma_v'(g_1,g_2):= \left(\frac{x(g_1g_2)}{x(g_1)}, \frac{x(g_1g_2)}{ x(g_2)}\right)_v, \quad \forall g_1,g_2 \in \SL_2(k_v).$$
Here
$$x\begin{pmatrix}a&b\\c&d\end{pmatrix}:= \begin{cases}c, & \text{ if $c \neq 0$,}\\ d, & \text{ if $c = 0$;} \end{cases}$$
and $(\cdot,\cdot)_v$ is the Hilbert quadratic symbol at $v$. 
Define a map $s_v: \SL_2(k_v)\rightarrow \{\pm1\}$ by setting
$$s_v\begin{pmatrix}a&b\\c&d\end{pmatrix}:=
\begin{cases}
(c,d)_v, & \text{ if $\ord_v(c)$ is odd and $d \neq 0$,} \\
1, & \text{ otherwise.} 
\end{cases}$$
Let $\sigma_v$\index{Kubota $2$-cocycle $\sigma_v$} be the $2$-cocycle defined by
$$\sigma_v(g_1,g_2):= \sigma_v'(g_1,g_2) s_v(g_1)s_v(g_2)s_v(g_1g_2)^{-1},
\quad \forall g_1,g_2 \in \SL_2(k_v).$$
It is known that (cf.\ \cite[Section 2.3]{Gel})
$\sigma_v(\kappa_1,\kappa_2) = 1 \ \forall \kappa_1,\kappa_2 \in \SL_2(O_v)$.
Hence $\sigma_v$ induces a central extension $\widetilde{\SL}_2(k_v)$ of $\SL_2(k_v)$ by $\{\pm 1\}$ which splits on the subgroup $\SL_2(O_v)$.
More precisely, the extension $\widetilde{\SL}_2(k_v)$ is identified with $\SL_2(k_v)\times \{\pm 1\}$ (as sets) with the following group law:
$$(g_1,\xi_1)\cdot (g_2,\xi_2) = \big(g_1g_2, \xi_1\xi_2 \sigma_v(g_1,g_2)\big).$$

Globally, we define a $2$-cocycle $\sigma$ on $\SL_2(k_\AA)$ by setting $\sigma:= \otimes_v \sigma_v$, and let $\widetilde{\SL}_2(k_\AA)$ be the corresponding central extension of $\SL_2(k_\AA)$ by $\{\pm 1\}$.
The section
$$\begin{tabular}{ccc}
$\SL_2(k_{\AA})$ & $\longrightarrow$ & $\widetilde{\SL}_2(k_\AA)$ \\
$\kappa$ & $\longmapsto$ & $(\kappa,1)$
\end{tabular}$$
becomes a group homomorphism when restricting to $\SL_2(O_\AA)$,
which embeds $\SL_2(O_\AA)$ into $\widetilde{\SL}_2(k_\AA)$ as a subgroup.
Moreover, for every $\gamma \in \SL_2(k)$, the value $s(\gamma) := \prod_v s_v(\gamma)$ is well-defined, and 
the embedding
$$\begin{tabular}{rccc}
$\SL_2(k)$ & $\longrightarrow$ & $\widetilde{\SL}_2(k_\AA)$ \\
$\gamma$ & $\longmapsto$ & $\big(\gamma,s(\gamma)\big)$
\end{tabular}$$
preserves the group law. Thus we may view $\SL_2(k)$ as a discrete subgroup of $\widetilde{\SL}_2(k_\AA)$.

\subsection{Weil representation}\label{sec B.2}

Let $\Dcal$ be a division quaternion algebra over $k$. We first write $\Dcal$ as $V_1 \oplus V_3$, where $V_1 = k$ and $V_3 := \{ b \in \Dcal: \tr_{\Dcal/k}(b) = 0\}$. Put $Q_{V_i} := \Nr_{\Dcal/k}\mid_{V_i}$. Then the quadratic space $(V_i,Q_{V_i})$ is anisotropic with dimension $i$, and $\text{SO}(V_3) \cong \Dcal^\times /k^\times$.
Let $\omega^{V_i} = \otimes_v \omega^{V_i}_v$ be the Weil representation of the metaplectic group $\widetilde{\SL_2}(k_\AA) \times \Oo(V_i)$ on the Schwartz space  $S(V_i(k_\AA))$, where for each place $v$, $\omega^{V_i}_v$ is defined as follows  (cf.\ \cite[Section 2.3]{Gel}):
\begin{eqnarray}
(1) && \omega^{V_i}_v(h) \phi(x):= \phi(h^{-1} x), \text{ $h \in \Oo(V_i)$}; \nonumber \\
(2) && \omega^{V_i}_v(1,\xi) \phi(x):= \xi \cdot \phi(x), \text{ $\xi \in \{\pm 1\}$}; \nonumber \\
(3) && \omega^{V_i}_v \left(\begin{pmatrix}1&u\\0&1\end{pmatrix},1\right)\phi(x) = \psi_v(u Q_{V_i}(x))\phi(x), \text{ $u \in k_v$};\nonumber \\
(4) && \omega^{V_i}_v \left(\begin{pmatrix}a_v&0\\0&a_v^{-1}\end{pmatrix},1\right)\phi(x)=|a_v|_v^{\frac{i}{2}} (a_v,a_v)_v \frac{\varepsilon^{V_i}_v(a_v)}{\varepsilon^{V_i}_v(1)} \cdot \phi(a_v x), \text{ $a_v \in k_v^{\times}$;} \nonumber \\
(5) && \omega^{V_i}_v\left(\begin{pmatrix}0&1\\-1&0\end{pmatrix},1\right)\phi(x)= \varepsilon^{V_i}_v(1) \cdot \widehat{\phi}(x). \nonumber
\end{eqnarray}
Here:
\begin{itemize}
\item 
$$\varepsilon^{V_i}_v(a_v):= \int_{L_{v}} \psi_v(a_v Q_{V_i}(x))d_{a_v}x, \quad \forall a_v \in k_v^{\times},$$
where $L_{v}$ is a sufficiently large $O_v$-lattice in $V_i(k_v)$, and the Haar measure $d_{a_v}x$ is self-dual with respect to the pairing 
$$(x,y)\mapsto \psi_v(a_v \cdot \tr_{\mathcal{D}/k}(x\bar{y})),\quad \forall x,y \in V_i(k_v);$$
\item $\widehat{\phi}(x)$ is the Fourier transform of $\phi$:
$$\widehat{\phi}(x) :=\int_{V_i(k_v)} \phi(y)\psi_v(\tr_{\mathcal{D}/k}(x\bar{y}))dy.$$
\end{itemize}

For $a = (a_v)_v \in k_\AA^\times$, we put $\varepsilon^{V_i}(a):= \prod_v \varepsilon_v^{V_i}(a_v)$.


\subsection{Siegel-Eisenstein series}\label{sec B.3}

Recall that $\B^1$ denotes the standard Borel subgroup of  $\SL_2$, 
and let $\widetilde{\B}^1$ be the preimage of $\B^1$ in $\widetilde{\SL}_2(k_\AA)$.

For each $\phi \in S(V_3(k_\AA))$, the Siegel section associated to $\phi$ is defined by:
$$\Phi_{\phi}(\tilde{g},s) = \xi \cdot \frac{\varepsilon^{V_3}(a)}{\varepsilon^{V_3}(1)} \cdot |a|_{k_\AA}^{s+\frac{1}{2}} \cdot \big(\omega^{V_3}(\kappa^1)\phi\big)(0), \quad \text{ for } \tilde{g} \in \widetilde{\SL}_2(k_\AA), 
$$
where $a \in k_\AA^\times$ and  $\kappa^1 \in \SL_2(O_\AA)$ so that $\tilde{g} =\left(\begin{pmatrix} a& * \\ 0 & a^{-1} \end{pmatrix},\xi\right) \cdot \kappa^1$.
The Siegel-Eisenstein series associated to $\varphi$ is then defined by
$$E(\tilde{g},s,\phi) := \sum_{\gamma \in B^1(k)\backslash \SL_2(k)} \Phi_{\phi}(\gamma \tilde{g},s),$$
which converges absolutely for $\re(s) > 3/2$ and has meromorphic continuation to the whole complex $s$-plane.
In this section, we shall verify that:

\begin{thm}\label{thm B.1}
Given $\phi \in S(V_3(k_\AA))$, we have
$$E(\tilde{g},1 , \phi) = \frac{1}{2} \cdot I^3(\tilde{g},\phi),  \quad \forall \tilde{g} \in \widetilde{\SL}_2(k_\AA),$$
where
$$I^3(\tilde{g},\phi) := \int_{\Dcal^\times k_\AA^\times/\Dcal_\AA^\times} \Big(\sum_{x \in V_3(k)} \omega^{V_3}(\tilde{g},b)\phi(x)\Big) db.$$
\end{thm}

\begin{rem}\label{rem B.2}
Write $(V_\Dcal,Q_{V_\Dcal}) = (V_1,Q_{V_1}) \oplus (V_3, Q_{V_3})$. Given $\varphi \in S(V_\Dcal(k_\AA))$, 
suppose $\varphi = \phi_1 \oplus \phi_3$ where $\phi_1 \in S(V_1(k_\AA))$ and $\phi_3 \in S(V_3(k_\AA))$.
Then $$I^\Dcal(g,s,\varphi) = \theta^k((g,1),\phi_1) \cdot E((g,1),s,\phi_3), \quad \forall g \in \SL_2(k_\AA),$$
where $\theta^k((g,1),\phi_1):= \sum_{x \in k} \big(\omega^{V_1}(g)\phi_1\big)(x)$.
On the other hand,
$$\theta^k((g,1),\phi_1) \cdot I^3((g,1),\phi_3) = \int_{\Dcal^\times k_\AA^\times/\Dcal_\AA^\times} \theta^{V_\Dcal}(g,[b,b];\varphi) db.$$
Therefore 
$$I^\Dcal(g,1,\varphi) = \frac{1}{2}\int_{\Dcal^\times k_\AA^\times/\Dcal_\AA^\times} \theta^{V_\Dcal}(g,[b,b];\varphi) db, \quad \forall g \in \SL_2(k_\AA) \text{ and } \varphi \in S(V_\Dcal(k_\AA)).$$
In particular, for $\phi \in \Pi$ we have
\begin{eqnarray}\label{eqn B.1}
J^\Dcal(1;\phi,\varphi)& = & \int_{\SL_2(k)\backslash \SL_2(k_\AA)} \phi(g) I^\Dcal(g,1,\varphi) dg \nonumber \\
&=& \frac{1}{2} \cdot \int_{\Dcal^\times k_\AA^\times \backslash \Dcal_\AA^\times} \theta^\Dcal(b,b;\phi,\varphi) db. 
\end{eqnarray}
\end{rem}
${}$

\subsubsection*{Proof of \text{\rm Theorem B.1}}
The proof of the above theorem basically follows the approach in \cite{Wei} for the even dimensional case. Here we recall the strategy as follows:
For each $\beta \in k$, the $\beta$-th Fourier coefficient of a given metaplectic form $f$ is:
$$f_\beta^*(\tilde{g}) := \int_{k \backslash k_\AA} f\left(\Big(\begin{pmatrix} 1 & u \\ 0 & 1\end{pmatrix},1\Big) \tilde{g}\right)\psi(-\beta u) du.$$
Given $\phi \in S(V_3(k_\AA))$, we verify that:
\begin{itemize}
\item[(1)] The equality holds for the \lq\lq constant terms,\rq\rq\ i.e.\ 
$$E^*_0(\tilde{g},1,\phi) = \big(\omega^{V_3}(\tilde{g})\phi\big)(0) = \frac{1}{2}\cdot I^{3,*}_0(\tilde{g},\phi).$$
In particular, this says that $E(\tilde{g},1 ,\phi) - 1/2 \cdot I^3(\tilde{g},\phi)$ is a cusp form on $\widetilde{\SL}_2(k_\AA)$.
\item [(2)] It is straightforward that $E(\cdot,s , \phi)$ is orthogonal to all the cuspidal metaplectic forms on $\widetilde{\SL}_2(k_\AA)$ with respect to the Petersson inner product.
\item [(3)] The theta integral $I^3(\cdot,\phi)$ is orthogonal to all the cuspidal metaplectic forms on $\widetilde{\SL}_2(k_\AA)$. Indeed, adapting the proof of \cite[Theorem A.4]{Wei}, there exists a constant $C$ so that 
$$I^{3,*}_\beta(\tilde{g},\phi) = C \cdot E^*_\beta (\tilde{g},1,\phi), \quad \forall \beta \neq 0.$$
In particular, the function $I'(\cdot, \phi):= I^3(\cdot, \phi)- C\cdot E(\cdot,1,\phi)$ satisfies
$$I'\left(\Big(\begin{pmatrix} 1&u\\0&1\end{pmatrix},1\Big)\tilde{g},\phi\right) = I'(\tilde{g},\phi), \quad \forall u \in k_\AA.$$
Thus $I'(\cdot,\phi)$ is orthogonal to all the cuspidal metaplectic forms, and so is $I^3(\cdot,\phi)$ by $(2)$.
\end{itemize}
Since a cusp form orthogonal to itself must be zero, we get $E(\tilde{g},1,\phi) = 1/2 \cdot I^3(\tilde{g},\phi)$ for every $\tilde{g} \in \widetilde{\SL}_2(k_\AA)$. 
\hfill $\Box$

\section{The case when $K = k\times k$}\label{sec C}

When $K = k \times k$, we have $L(s,\varsigma_K) = \zeta_k(s)$. Moreover, the existance of an embedding $\iota: K\hookrightarrow \Dcal$ implies that $\Dcal = \Mat_2(k)$, and $\Pi^\Dcal = \Pi$. We may write the given Hecke character $\chi$ as $\chi_1 \times \chi_2$, where $\chi_i$ is a unitary Hecke character on $k^\times \backslash k_\AA^\times$ for $i = 1,2$.
The assumption $\eta \cdot \chi\big|_{k_\AA^\times} = 1$ says that $\Pi \otimes \chi_2 = \widetilde{\Pi}\otimes \chi_1^{-1}$. Thus we have $L(s,\Pi\times \chi) = L(s,\Pi \otimes \chi_1) \cdot L(s,\Pi\otimes \chi_2)$.

Without lose of generality, suppose $\iota(a,b) = \begin{pmatrix}a&0\\0&b\end{pmatrix}$ for every $a,b \in k$. Then for $f \in \Pi$ and $\tilde{f} \in \widetilde{\Pi}$,
$$\Pcal_\chi^\Dcal(f\otimes \tilde{f}) = Z(\frac{1}{2}; f,\chi_1) \cdot Z(\frac{1}{2}; \tilde{f},\chi_1^{-1}),$$
where
\begin{eqnarray}
&& Z(s;f,\chi_1) := \int_{k^\times \backslash k_\AA^\times} f\begin{pmatrix} y & 0 \\0&1\end{pmatrix} \chi_1(y) |y|_\AA^{s-1/2} d^\times y \nonumber \\
\text{ and } && Z(s;\tilde{f},\chi_1^{-1}) := \int_{k^\times \backslash k_\AA^\times} \tilde{f}\begin{pmatrix} y & 0 \\0&1\end{pmatrix} \chi_1^{-1}(y) |y|_\AA^{s-1/2} d^\times y \nonumber
\end{eqnarray}
are entire functions on the complex $s$-plane.
Note that for a pure tensor $f = \otimes_v f_v \in \Pi$, one has (cf.\ \cite[the equality~(5.31) in Chapter 3 and Proposition 3.5.3]{Bum})
\begin{eqnarray}\label{eqn C.1}
Z(s;f,\chi_1) &=& L(s,\Pi \otimes \chi_1) \cdot \prod_v Z_v^o(s;f_v,\chi_{1,v}),
\end{eqnarray}
where for each place $v$ of $k$, put
$$Z_v^o(s;f_v,\chi_{1,v}) := \frac{1}{L_v(s,\Pi\otimes\chi_1)} \cdot \int_{k_v^\times} W_{f_v}\begin{pmatrix} y_v & 0 \\ 0& 1\end{pmatrix} \chi_{1,v}(y_v)|y_v|_v^{s-1/2} d^\times y_v.$$
Here $W_{f_v}$ is the Whittaker function associated to $f_v$ with respect to the chosen $\psi_v$.
Similarly, for a pure tensor $\tilde{f} = \otimes \tilde{f}_v \in \widetilde{\Pi}$ we have
\begin{eqnarray}\label{eqn C.2}
Z(s;\tilde{f},\chi_1^{-1}) &=& L(s,\widetilde{\Pi} \otimes \chi_1^{-1}) \cdot \prod_v Z_v^o(s;\tilde{f}_v,\chi_{1,v}^{-1}),
\end{eqnarray}
where
$$Z_v^o(s;\tilde{f}_v,\chi_{1,v}^{-1}) = \frac{1}{L_v(s,\widetilde{\Pi}\otimes\chi_1^{-1})} \cdot \int_{k_v^\times} W_{\tilde{f}_v}'\begin{pmatrix} y_v & 0 \\ 0& 1\end{pmatrix} \chi_{1,v}^{-1}(y_v)|y_v|_v^{s-1/2} d^\times y_v.$$
Here we take $W_{\tilde{f}_v}'$ to be the Whittaker function associated to $\tilde{f}_v$ with respect to $\overline{\psi}_v$.
Note that the validity of Ramanujan bounds for $\Pi$ and $\widetilde{\Pi}$ implies that $Z_v^o(s;f_v,\chi_{1,v})$ and $Z_v^o(s;\tilde{f}_v,\chi_{1,v}^{-1})$ both converge absolutely at $s = 1/2$.

Recall that we may choose $\langle \cdot,\cdot \rangle_v^{\Mat_2} : \Pi_v \times \widetilde{\Pi}_v \rightarrow \CC$ by:
$$\langle f_v,\tilde{f}_v\rangle_v^{\Mat_2} := \frac{\zeta_v(2)}{\zeta_v(1) L_v(1,\Pi, \text{Ad})} \cdot \int_{k_v^\times} W_{f_v}\begin{pmatrix} y & 0 \\ 0&1 \end{pmatrix} W_{\tilde{f}_v}'\begin{pmatrix} y & 0 \\ 0&1 \end{pmatrix} d^\times y.$$
Indeed, by the Rankin-Selberg method these local pairings satisfy:
$$\langle \cdot, \cdot\rangle_{\text{Pet}}^{\Mat_2} = \frac{2 L(1,\Pi,\text{Ad})}{\zeta_k(2)} \cdot \prod_v \langle \cdot,\cdot \rangle_v^{\Mat_2}.$$
Therefore we get
$$\Pscr_{\chi,v}^\Dcal(f_v \otimes \tilde{f}_v) = Z_v^o(\frac{1}{2}; f_v , \chi_{1,v}) \cdot Z_v^o(\frac{1}{2};\tilde{f}_v,\chi_{1,v}^{-1}).$$
From the equation~(\ref{eqn C.1}) and (\ref{eqn C.2}), we arrive at:
\begin{thm}\label{thm C.1}
\text{\rm Theorem~\ref{thm 0.1}} holds for the case when $K = k\times k$.
\end{thm}

\subsubsection*{Acknowledgments}

The authors are grateful to Jing Yu for his steady encouragements. This work was completed while the second author was visiting Institute for Mathematical Sciences at National University of Singapore. He would like to thank Professor Wee Teck Gan for the invitation, and the institute for kind hospitality and wonderful working conditions.

\end{document}